\documentclass[11pt,a4paper]{amsart}
\usepackage{amsthm, amsfonts, amssymb, amsmath, latexsym, enumerate, array, 
amscd}
\usepackage[all]{xy}
\usepackage{hyperref,cite,mdwlist,mathrsfs}

\theoremstyle{plain}

\newtheorem{thm}{Theorem}[section] \newtheorem{lem}[thm]{Lemma}

\newtheorem{corol}[thm]{Corollary} \newtheorem{prop}[thm]{Proposition}

\newtheorem*{thm*}{Theorem}

\theoremstyle{definition} \newtheorem{rmk}[thm]{Remark}
 \newtheorem{dfn}{Definition}[section]

\newtheorem*{rmk*}{Remark}

\numberwithin{equation}{section}

\usepackage{anysize}
\marginsize{3.0cm}{3.0cm}{3.9cm}{3.9cm}

\begin{document}

\title[Vector bundles on quadric 3-folds]{A refined stable restriction theorem 
for vector bundles on quadric threefolds}

\author[Coand\u{a}]{Iustin~Coand\u{a}}
\address{Institute of Mathematics of the Romanian Academy, P.O.~Box 1--764,  
\newline         RO--014700 Bucharest, Romania}
\email{Iustin.Coanda@imar.ro}

\author[Faenzi]{Daniele~Faenzi}
\address{Universit\'{e} de Pau et des Pays de l'Adour, 
         Avenue de l'Universit\'{e}, BP 576, 
\newline         64012 Pau Cedex, France}
\email{daniele.faenzi@univ-pau.fr}

\begin{abstract}
Let $E$ be a stable rank 2 vector bundle on a smooth quadric threefold 
$Q$ in the projective 4-space $P$. We show that the hyperplanes $H$ in $P$ 
for which the restriction of $E$ to the hyperplane section of $Q$ by $H$ is 
not stable form, in general, a closed subset of codimension at least 2 of 
the dual projective 4-space, and we explicitly describe the bundles $E$ which 
do not enjoy this property. This refines a restriction theorem of Ein and 
Sols [\emph{Nagoya Math. J.} 96, 11--22 (1984)] in the same way the main 
result of Coand\u{a} [\emph{J. reine angew. Math.} 428, 97--110 (1992)] 
refines the restriction theorem of Barth [\emph{Math. Ann.} 226, 125--150 
(1977)]. 
\end{abstract}  



\maketitle

\section{Introduction}

Let $E$ be a stable rank 2 vector bundle on the quadric threefold 
$Q = Q_3 \subset {\mathbb P}^4$, with $c_1(E) = 0$ or $-1$ and with $c_2(E) = 
c_2[L]$, where $[L]$ is the cohomology class of a line $L \subset Q$ and 
$c_2 \in {\mathbb Z}$. As it is well known, if $c_1 = 0$ then $c_2 \geq 2$ and 
it is even, and if $c_1 = -1$ then $c_2 \geq 1$. Moreover, if $c_1 = -1$ and 
$c_2 = 1$ then $E$ is isomorphic to the {\it spinor bundle} $\mathcal S$. 
Examples of such bundles can be obtained, for $c_1 = 0$, by considering 
extensions: 
\begin{equation}
\label{extconics}
0 \longrightarrow {\mathcal O}_Q(-1) \longrightarrow E \longrightarrow 
{\mathcal I}_Y(1) \longrightarrow 0
\end{equation}
where $Y$ is a disjoint union of $c_2/2 + 1$ conics contained in $Q_3$, and 
for $c_1 = -1$ by considering extensions: 
\begin{equation}
\label{extlines}
0 \longrightarrow {\mathcal O}_Q(-1) \longrightarrow E \longrightarrow 
{\mathcal I}_Y \longrightarrow 0
\end{equation}
where $Y$ is a disjoint union of $c_2$ lines. For $c_1 = 0$ or $-1$ and 
$c_2 = 2$, any stable rank 2 vector bundle $E$ on $Q$ can be obtained in this 
way (see \cite[Cor.~of~Prop.~1]{SSW91} and \cite[Prop.~4.4]{OS94}).  
Alternatively, for $c_1 = 0$ and 
$c_2 = 2$, any such bundle $E$ is the pull-back of a null-correlation bundle on 
${\mathbb P}^3$ by a linear projection $Q \rightarrow {\mathbb P}^3$ with 
centre a point of ${\mathbb P}^4 \setminus Q$.  
Let ${\mathbb P}^{4\vee}$ be the 
dual projective space parametrizing the hyperplanes in ${\mathbb P}^4$ and 
let $Q^{\vee} \subset {\mathbb P}^{4\vee}$ be the dual quadric parametrizing 
the tangent hyperplanes to $Q_3$. For $h \in {\mathbb P}^{4\vee}$, let us 
denote by $H$ the corresponding hyperplane of ${\mathbb P}^4$. Ein and Sols 
\cite[Thm.~1.6]{ES84} showed that, for a general $h \in {\mathbb P}^{4\vee} 
\setminus Q^{\vee}$, the restriction $E\vert_{H\cap Q}$ is stable on 
$H\cap Q \simeq {\mathbb P}^1 \times {\mathbb P}^1$. However, {\it they 
missed an exception}, namely the spinor bundle $\mathcal S$. 
Actually, the two authors, guided (probably) by the case of 
${\mathbb P}^3$ where the exceptions appear for $c_1 = 0$,  
worked out the details for the case $c_1 = 0$ of their result  
and left the (similar) details for the case $c_1 = -1$ to the reader. We shall 
provide, in Remark~\ref{R:correctioneinsols}, an extra argument 
showing that the spinor bundle is, in fact, the only exception. 

In this paper we prove the following refinement of the result of Ein and 
Sols (analogous to the refinement of the restriction theorem of Barth 
\cite{Ba77} from \cite{Co92}): 

\begin{thm}\label{T:nonstblhyp}
Let $E$ be as above and let $\Sigma \subset {\mathbb P}^{4\vee}$, 
$\Sigma \neq Q^{\vee}$, be an irreducible hypersurface. If, for a general 
point $h \in \Sigma \setminus Q^{\vee}$, $E\vert_{H\cap Q}$ is not stable 
then either:   

\emph{(i)} $c_1(E) = 0$, $\Sigma$ consists of the hyperplanes passing 
through a 
point $x \in {\mathbb P}^4 \setminus Q$, and $E$ is the pull-back of a 
nullcorrelation bundle on ${\mathbb P}^3$ by the linear projection of 
centre $x$ restricted to $Q \rightarrow {\mathbb P}^3$, or: 

\emph{(ii)} $c_1(E) = -1$, $E$ can be realized as an extension 
(\ref{extlines})  with all the 
components of $Y$ contained in a smooth hyperplane section $H_0 \cap Q$ of 
$Q$ and $\Sigma = (H_0 \cap Q)^{\vee}$. 
\end{thm}     

The proof follows the strategy from \cite{Co92}. It is based on a variant of 
the so called Standard Construction which is explained in 
Section~\ref{S:standard}. The sections~\ref{S:lines} and~\ref{S:auxq2} contain 
some auxiliary results, needed in the proof of Thm.~\ref{T:nonstblhyp} which 
is given in Section~\ref{S:proof}. We close Section~\ref{S:proof} with a 
remark pointing out an important simplification in the proof of the main 
result of \cite{Co92}. 

Finally, we study, in Section~\ref{S:proof2}, the restrictions of $E$ (or, 
more generally, of a stable rank 2 reflexive sheaf) to the singular hyperplane 
sections of $Q$ and prove, using the same method, the following: 

\begin{thm}\label{T:singnonstblhyp}
Let $\mathcal E$ be a stable rank $2$ reflexive sheaf on $Q$, with 
$c_1(\mathcal E) = 0$ or $-1$. Then, for a general point $y \in Q$ such that, 
in particular,  
the tangent hyperplane ${\fam0 T}_yQ$ contains no singular point of 
$\mathcal E$, the restriction of $\mathcal E$ to $Y := Q \cap {\fam0 T}_yQ$ 
is stable, unless $\mathcal E$ is isomorphic to the spinor bundle 
$\mathcal S$. 
\end{thm} 

One derives immediately, from the above two theorems, the following:  

\begin{corol}\label{C:nonstblhyp} 
Let $E$ be a stable rank $2$ vector bundle on $Q$, with $c_1(E) = 0$ or 
$-1$. Then, except for the case where $E$ is one of the bundles appearing in 
the conclusion of Theorem~\ref{T:nonstblhyp}, the set of the hyperplanes 
$H \subset {\mathbb P}^4$ for which $E\vert_{H \cap Q}$ is not stable has 
codimension $\geq 2$ in ${\mathbb P}^{4\vee}$.  
\end{corol}

The main result of \cite{Co92} has been used in the study of the moduli 
spaces of mathematical instanton bundles on ${\mathbb P}^3$: see, for example, 
the paper of Katsylo and Ottaviani \cite{KO03}. We hope that the results of 
the present paper might have applications in the study of the moduli spaces 
of {\it odd instanton bundles} on $Q_3 \subset {\mathbb P}^4$, introduced 
recently by Faenzi \cite{Fa11}. 

\subsection*{Notation and conventions} (i) We work only with 
quasi-projective schemes over the field $\mathbb C$ of complex numbers. By a 
{\it point} we always mean a {\it closed point}. 

(ii) If $E$ is a rank $n$ vector bundle (= locally free sheaf) on a scheme 
$X$, we denote by ${\mathbb G}_r(E)$ (resp., ${\mathbb G}^r(E)$) the relative 
Grassmannian of rank $r$ subbundles (resp., quotient bundles) of $E$. Of 
course, ${\mathbb G}_r(E) \simeq {\mathbb G}^{n-r}(E)$ and 
${\mathbb G}_r(E) \simeq {\mathbb G}^r(E^{\ast})$. We use the classical 
convention (dual to Grothendieck's convention) for projective bundles, 
namely ${\mathbb P}(E) := {\mathbb G}_1(E)$. We shall also use the 
notation: ${\mathbb G}_a({\mathbb P}^n) := 
{\mathbb G}_{a+1}({\mathbb C}^{n+1})$. 

(iii) If $X \subset {\mathbb P}^n$ is a non-singular, connected projective 
variety we denote by $X^{\vee}$ its dual variety, i.e., the set of points 
$h$ of the dual projective space ${\mathbb P}^{n\vee}$ with the property that 
the corresponding hyperplane $H \subset {\mathbb P}^n$ contains the tangent 
linear subspace $\text{T}_xX$ to $X$ at some point $x \in X$. In particular, 
if $L$ (resp., $x$) is a linear subspace (resp., point) of $X$, then 
$L^{\vee}$ (resp., $x^{\vee}$) consists of the points 
$h \in {\mathbb P}^{n\vee}$ such that $H \supset L$ (resp., $H \ni x$). 
$L^\vee$ is a linear subspace of ${\mathbb P}^{n\vee}$.  

(iv) When we say that a sheaf is ``(semi)stable'' we mean that it is 
(semi)stable in the sense of Mumford and Takemoto (or $\mu$-(semi)stable, or 
slope (semi)stable) with respect to a polarization which should be obvious in 
each of the cases under consideration. In particular, if $Q_2 \subset 
{\mathbb P}^3$ is a nonsingular quadric surface then we use the polarization 
${\mathcal O}_{Q_2}(1,1)$.        

(v) If $Y$ is a closed subscheme of a scheme $X$, we shall denote by 
${\mathcal I}_{Y,X}$ the kernel of the canonical epimorphism 
${\mathcal O}_X \rightarrow {\mathcal O}_Y$, i.e., the ideal sheaf of 
${\mathcal O}_X$ defining $Y$ as a closed subscheme of $X$. 

(vi) If $f : X \rightarrow Y$ is a morphism of schemes, we shall denote, 
occasionally, the sheaf ${\Omega}_{X/Y}$ of relative K\"{a}hler differentials 
by ${\Omega}_f$.

\section{The Standard Construction}\label{S:standard}

\begin{dfn}\label{D:secondform}
Let $p : X \rightarrow Y$ be a morphism of schemes, let $\mathcal F$ be a 
coherent sheaf on $Y$ and let: 
\begin{equation}
\label{exactseqf}
0 \longrightarrow {\mathcal F}^{\prime} \longrightarrow p^{\ast}{\mathcal F} 
\longrightarrow {\mathcal F}^{\prime \prime} \longrightarrow 0
\end{equation}
be an exact sequence of coherent sheaves on $X$. Consider the fibre product 
$X\times_YX$, the projections $p_1,p_2 : X\times_YX \rightarrow X$ and let 
$\Delta = {\Delta}_{X/Y} \subset X\times_YX$ be the image of the diagonal 
embedding $\delta = {\delta}_{X/Y} : X \rightarrow X\times_YX$. The composite 
morphism: 
\begin{equation}
\label{composition}
p_1^{\ast}{\mathcal F}^{\prime} \longrightarrow p_1^{\ast}p^{\ast}{\mathcal F} = 
p_2^{\ast}p^{\ast}{\mathcal F} \longrightarrow 
p_2^{\ast}{\mathcal F}^{\prime \prime}
\end{equation}
vanishes along $\Delta$, hence the composite morphism:
\begin{equation*}
p_1^{\ast}{\mathcal F}^{\prime} \longrightarrow 
p_2^{\ast}{\mathcal F}^{\prime \prime} \longrightarrow 
p_2^{\ast}{\mathcal F}^{\prime \prime}\vert_\Delta 
\end{equation*}
is 0. But the sequence: 
\begin{equation*}
0 \longrightarrow {\mathcal I}_{\Delta}\otimes    
p_2^{\ast}{\mathcal F}^{\prime \prime} \longrightarrow 
p_2^{\ast}{\mathcal F}^{\prime \prime} \longrightarrow 
p_2^{\ast}{\mathcal F}^{\prime \prime}\vert_\Delta \longrightarrow 0
\end{equation*}
is exact (also to the left). [{\it Indeed}, we may assume that 
$Y = \text{Spec}\, A$, $X = \text{Spec}\, B$ and ${\mathcal F} = 
{\widetilde M}$. The sequence: 
\begin{equation*}
0 \longrightarrow I \longrightarrow B\otimes_AB \longrightarrow B 
\longrightarrow 0
\end{equation*}
is a split exact sequence of right $B$-modules and if $N$ is a 
$B\otimes_AB$-module then 
\begin{equation*}
N\otimes_{B\otimes_AB}(B\otimes_AM) \simeq 
N\otimes_BM, 
\end{equation*}
hence the sequence: 
\begin{equation*}
0 \longrightarrow I\otimes_{B\otimes_AB}(B\otimes_AM) \longrightarrow 
B\otimes_AM \longrightarrow B\otimes_{B\otimes_AB}(B\otimes_AM) 
\longrightarrow 0
\end{equation*}
is exact.] One deduces that the morphism (\ref{composition}) induces a 
morphism $p_1^{\ast}{\mathcal F}^{\prime} \rightarrow 
{\mathcal I}_{\Delta}\otimes p_2^{\ast}{\mathcal F}^{\prime \prime}$ which 
restricted to $\Delta$ gives us a morphism ${\mathcal F}^{\prime} \rightarrow 
{\Omega}_{X/Y}\otimes_{{\mathcal O}_X}{\mathcal F}^{\prime \prime}$. This 
morphism is called the {\it second fundamental form} of the exact sequence 
(\ref{exactseqf}). 
\end{dfn}

\begin{lem}\label{L:diffgrass}
Let $p : X \rightarrow Y$ be a morphism of schemes, let $\mathcal E$ be a 
locally free ${\mathcal O}_Y$-module and let 
$0 \rightarrow {\mathcal E}^{\prime} \rightarrow 
p^{\ast}{\mathcal E} \rightarrow 
{\mathcal E}^{\prime \prime} \rightarrow 0$
be an exact sequence of locally free sheaves on $X$. Let $r$ be the rank of 
${\mathcal E}^{\prime \prime}$, let $\pi : {\mathbb G} = 
{\mathbb G}^r(\mathcal E) \rightarrow Y$ be the relative Grassmannian of rank 
$r$ quotients of $\mathcal E$ and let $f : X \rightarrow {\mathbb G}$ be the 
$Y$-morphism corresponding to the epimorphism $p^{\ast}{\mathcal E} 
\rightarrow {\mathcal E}^{\prime \prime}$. Then the composite morphism 
\begin{equation*}
{\mathcal Hom}_{{\mathcal O}_X}({\mathcal E}^{\prime \prime},
{\mathcal E}^{\prime})\longrightarrow 
{\mathcal Hom}_{{\mathcal O}_X}({\mathcal E}^{\prime \prime},
{\Omega}_{X/Y}\otimes {\mathcal E}^{\prime \prime}) 
\overset{\fam0 Tr}{\longrightarrow} {\Omega}_{X/Y}
\end{equation*}
deduced from the second fundamental form of the above exact sequence can be 
identified with the relative differential ${\fam0 d}f : f^{\ast}
{\Omega}_{{\mathbb G}/Y} \rightarrow {\Omega}_{X/Y}$ of $f$. 
\end{lem}

\begin{proof}
We follow the argument from the proof of \cite[Cor.~2.2.10]{HL97}. Let 
$\mathcal B$ be the universal quotient of ${\pi}^{\ast}{\mathcal E}$ and let 
$0 \rightarrow {\mathcal A} \rightarrow {\pi}^{\ast}{\mathcal E} \rightarrow 
{\mathcal B} \rightarrow 0$ be the tautological exact sequence on 
$\mathbb G$. Let $q_1,q_2 : {\mathbb G}\times_Y{\mathbb G} \rightarrow 
{\mathbb G}$ be the canonical projections and let ${\mathcal J} \subset 
{\mathcal O}_{{\mathbb G}\times_Y{\mathbb G}}$ be the ideal sheaf of the diagonal 
${\Delta}_{{\mathbb G}/Y} \subset {\mathbb G}\times_Y{\mathbb G}$. Consider the 
morphism $q_1^{\ast}{\mathcal A} \rightarrow q_2^{\ast}{\mathcal B}$ analogous 
to the morphism (\ref{composition}) from Def.~\ref{D:secondform}. Using the 
universal property of the relative Grassmannian, one sees easily that, for 
every $Y$-scheme $Z$, an $Y$-morphism $\varphi : Z \rightarrow 
{\mathbb G}\times_Y{\mathbb G}$ factors through ${\Delta}_{{\mathbb G}/Y}$ if 
and only if the morphism ${\varphi}^{\ast}q_1^{\ast}{\mathcal A} \rightarrow 
{\varphi}^{\ast}q_2^{\ast}{\mathcal B}$ is 0. One deduces that 
${\Delta}_{{\mathbb G}/Y}$ is the zero scheme of the morphism 
$q_1^{\ast}{\mathcal A} \rightarrow q_2^{\ast}{\mathcal B}$, which means that 
the image of the composite morphism: 
\begin{equation*}
{\mathcal Hom}_{{\mathcal O}_{{\mathbb G}\times_Y{\mathbb G}}}(q_2^{\ast}{\mathcal B}, 
q_1^{\ast}{\mathcal A}) \longrightarrow 
{\mathcal Hom}_{{\mathcal O}_{{\mathbb G}\times_Y{\mathbb G}}}(q_2^{\ast}{\mathcal B}, 
q_2^{\ast}{\mathcal B}) \overset{\text{Tr}}{\longrightarrow} 
{\mathcal O}_{{\mathbb G}\times_Y{\mathbb G}}
\end{equation*}
is exactly $\mathcal J$. Restricting the epimorphism 
${\mathcal Hom}(q_2^{\ast}{\mathcal B},q_1^{\ast}{\mathcal A}) \rightarrow 
{\mathcal J}$ to ${\Delta}_{{\mathbb G}/Y}$ one gets an epimorphism 
${\mathcal Hom}_{{\mathcal O}_{\mathbb G}}({\mathcal B},{\mathcal A}) \rightarrow 
{\mathcal J}/{\mathcal J}^2 = {\Omega}_{{\mathbb G}/Y}$, which must be an 
isomorphism because its source and its target are locally free sheaves of the 
same rank. 

Now, one has only to recall that the relative differential $\text{d}f : 
f^{\ast}{\Omega}_{{\mathbb G}/Y} \rightarrow {\Omega}_{X/Y}$ can be identified 
with the morphism $(f\times_Yf)^{\ast}({\mathcal J}/{\mathcal J}^2) 
\rightarrow {\mathcal I}_{\Delta}/{\mathcal I}_{\Delta}^2$.      
\end{proof}

\begin{dfn}\label{D:satsubsheaf}
Let $\mathcal E$ be a locally free sheaf on a nonsingular, connected variety 
$X$. A coherent subsheaf ${\mathcal E}^{\prime}$ of $\mathcal E$ is called 
{\it saturated} if the quotient ${\mathcal E}/{\mathcal E}^{\prime}$ is 
torsion-free. This is equivalent to the fact that ${\mathcal E}^{\prime}$ is 
reflexive and ${\mathcal E}/{\mathcal E}^{\prime}$ is locally free outside a 
closed subset of $X$ of codimension $\geq 2$. See, for example, 
\cite[II,~Sect.~1.1]{OSS80} or \cite[Sect.~1]{Ha80}.  
\end{dfn} 

\begin{dfn}\label{D:genericzero}
We say that a morphism $\varphi : {\mathcal F} \rightarrow {\mathcal G}$ of 
coherent sheaves on a nonsingular, connected variety $X$ is {\it generically 
zero} if there exists a non-empty open subset $U$ of $X$ such that 
$\varphi \vert_U = 0$. 

Assume that such a morphism is not generically zero. There exists a 
non-empty open subset $U$ of $X$ such that ${\mathcal F}\vert_U$ and 
${\mathcal G}\vert_U$ are locally free. Since $\varphi \vert_U 
\neq 0$, it follows that there exists a non-empty open subset $U^{\prime}$ of 
$U$ such that $\varphi (x) : {\mathcal F}(x) \rightarrow {\mathcal G}(x)$ is 
non-zero, $\forall x \in U^{\prime}$ (here ${\mathcal F}(x) := 
{\mathcal F}_x/\mathfrak{m}_x{\mathcal F}_x$). We say, in this case, that 
$\varphi$ {\it has generically rank} $\geq 1$. 
\end{dfn}           

\begin{prop}\label{P:standconstr}
Let $p : X \rightarrow Y$ be a dominant morphism of nonsingular, connected 
varieties, with all the non-empty fibres of dimension ${\fam0 dim}\, X - 
{\fam0 dim}\, Y$, and with irreducible general fibres. Let $\mathcal E$ be a 
locally free sheaf on $Y$, let ${\mathcal E}^{\prime}$ be a saturated subsheaf 
of $p^{\ast}{\mathcal E}$ and let 
${\mathcal E}^{\prime \prime} := p^{\ast}{\mathcal E}/{\mathcal E}^{\prime}$. 

If the second fundamental form ${\mathcal E}^{\prime} \rightarrow 
{\Omega}_{X/Y}\otimes {\mathcal E}^{\prime \prime}$ associated to the short 
exact sequence $0 \rightarrow {\mathcal E}^{\prime} \rightarrow 
p^{\ast}{\mathcal E} \rightarrow {\mathcal E}^{\prime \prime} \rightarrow 0$ is 
generically zero then there exists a saturated subsheaf 
${\overline {\mathcal E}}^{\prime}$ of $\mathcal E$ such that 
${\mathcal E}^{\prime} = p^{\ast}{\overline {\mathcal E}}^{\prime}$. 
\end{prop}

\begin{proof}
We follow the classical approach, originating in Van de Ven~\cite{VdV72}, 
Grauert and M\"{u}lich~\cite{GM75} and Barth~\cite{Ba77} and clearly explained 
in Forster et al.~\cite{FHS80}. For another approach, see the proof of 
Prop. 1.11 in Flenner~\cite{Fl84}. 

Let $Z$ be a closed subset of $X$, of codimension $\geq 2$, such that 
${\mathcal E}^{\prime \prime}\vert_{X\setminus Z}$ is locally free. Let $r$ 
be the rank of ${\mathcal E}^{\prime \prime}$, let $\pi : 
{\mathbb G}^r(\mathcal E) \rightarrow Y$ be the relative Grassmannian of rank 
$r$ quotients of $\mathcal E$, and let $f : X\setminus Z \rightarrow 
{\mathbb G}^r(\mathcal E)$ be the $Y$-morphism defined by the epimorphism 
$p^{\ast}{\mathcal E}\vert_{X\setminus Z} \rightarrow 
{\mathcal E}^{\prime \prime}\vert_{X\setminus Z}$. From the hypothesis and 
from Lemma~\ref{L:diffgrass} it follows that the relative differential 
$\text{d}f : f^{\ast}{\Omega}_{{\mathbb G}/Y} \rightarrow 
{\Omega}_{X/Y}\vert_{X\setminus Z}$ is {\it generically zero}. 

Now, for a general point $y \in Y$, the fibre $X_y$ is smooth (by generic 
smoothness), irreducible (by hypothesis), with $\text{codim}(Z_y,X_y) \geq 2$, 
and with $\text{d}f\vert_{X_y\setminus Z_y} = 0$ hence $f(X_y\setminus 
Z_y)$ consists of a single point. Let $Y^{\prime}$ be the closure of 
$f(X\setminus Z)$ in ${\mathbb G}^r(\mathcal E)$. Since $f(X\setminus Z)$ 
contains an open dense subset of $Y^{\prime}$, one shows easily (using the 
well-known results about the dimension of the fibres of a morphism) that, 
for a general point $y\in Y$, the fibre $Y^{\prime}_y$ is the closure of 
$f(X_y\setminus Z_y)$, hence it consists of a single point. One deduces that 
$\pi\vert_{Y^\prime} : Y^{\prime} \rightarrow Y$ is {\it birational}. Now,  
Zariski's Main Theorem (in the elementary variant from 
Mumford~\cite[III,~\S~9,~Prop.~1]{Mu99}) implies that there exists an open 
subset $V = Y\setminus Z^{\prime}$ of $Y$, with $\text{codim}(Z^{\prime},Y) 
\geq 2$, such that the restriction of  
$\pi : {\pi}^{-1}(V) \cap Y^{\prime} \rightarrow 
V$ is an isomorphism. Let $\sigma : V \rightarrow {\mathbb G}^r(\mathcal E)$ 
be the section over $V$ of $\pi : {\mathbb G}^r(\mathcal E) \rightarrow Y$ 
defined by the inverse of this isomorphism. One has $f = \sigma \circ p$ over 
$X \setminus (Z \cup p^{-1}(Z^{\prime}))$. 

Let $\mathcal A$ be the kernel of the universal quotient 
${\pi}^{\ast}{\mathcal E} \rightarrow {\mathcal B}$. ${\sigma}^{\ast}
{\mathcal A}$ (which is a vector subbundle of ${\mathcal E}\vert_V$) 
can be extended to a saturated subsheaf ${\overline {\mathcal E}}^{\prime}$ of 
$\mathcal E$. From the hypothesis on the dimensions of the fibres of $p$ it 
follows, on one hand, that $\text{codim}(p^{-1}(Z^{\prime}),X) \geq 2$ and, 
on the other hand, that $p$ is {\it flat} (see 
Hartshorne~\cite[III,~Ex.~10.9]{Ha77}). One deduces that 
$p^{\ast}{\overline {\mathcal E}}^{\prime}$ is a saturated subsheaf of 
$p^{\ast}{\mathcal E}$. Finally, since ${\mathcal E}^{\prime}$ and 
$p^{\ast}{\overline {\mathcal E}}^{\prime}$ coincide over 
$X \setminus (Z \cup p^{-1}(Z^{\prime}))$ they must coincide over $X$.    
\end{proof}

\begin{lem}\label{L:relomega}
Let ${\mathbb G}_d({\mathbb P}^n)$ be the Grassmannian of $d$-dimensional 
linear subspaces of ${\mathbb P}^n$, $1 \leq d <n$, 
and consider the incidence diagram: 
\begin{equation*}
\begin{CD}
{\mathbb F}_{0,d}({\mathbb P}^n) @>q>> {\mathbb G}_d({\mathbb P}^n)\\
@VpVV \\
{\mathbb P}^n
\end{CD}
\end{equation*}
For a point $\ell \in {\mathbb G}_d({\mathbb P}^n)$, $p$ maps isomorphically 
the fibre $q^{-1}(\ell)$ onto the corresponding linear subspace $L$ of 
${\mathbb P}^n$. Then ${\Omega}_{{\mathbb F}/{\mathbb P}}\vert_{q^{-1}(\ell)}  
\simeq {\fam0 T}_L(-1)^{\oplus \, n-d}$.
\end{lem}

\begin{proof}
Consider the tautological exact sequences on ${\mathbb P}^n$ and 
${\mathbb G}_d({\mathbb P}^n)$: 
\begin{gather*}
0 \longrightarrow {\mathcal O}_{\mathbb P}(-1) \longrightarrow 
{\mathcal O}_{\mathbb P}^{\oplus \, n+1} \longrightarrow 
\text{T}_{\mathbb P}(-1) \longrightarrow 0\\
0 \longrightarrow {\mathcal A} \longrightarrow 
{\mathcal O}_{\mathbb G}^{\oplus \, n+1} \longrightarrow {\mathcal B} 
\longrightarrow 0
\end{gather*}
(with $\text{rk}\, {\mathcal A} = d+1$). The composite morphism 
$p^{\ast}{\mathcal O}_{\mathbb P}(-1) \rightarrow  
{\mathcal O}_{\mathbb F}^{\oplus \, n+1} \rightarrow q^{\ast}{\mathcal B}$ is 0, 
hence it induces an epimorphism $p^{\ast}\text{T}_{\mathbb P}(-1) \rightarrow 
q^{\ast}{\mathcal B}$. One can easily show that 
\begin{equation*}
{\mathbb F}_{0,d}({\mathbb P}^n) 
\simeq {\mathbb G}^{n-d}(\text{T}_{\mathbb P}(-1))\   \text{over}\   
{\mathbb P}^n 
\end{equation*}
such that 
the universal quotient of $p^{\ast}\text{T}_{\mathbb P}(-1)$ corresponds to 
$q^{\ast}{\mathcal B}$. Let ${\mathcal A}^{\prime}$ be the kernel of the 
epimorphism $p^{\ast}\text{T}_{\mathbb P}(-1) \rightarrow q^{\ast}{\mathcal B}$. 
Restricting to $q^{-1}(\ell)$ the exact sequence: 
\begin{equation*}
0 \longrightarrow p^{\ast}{\mathcal O}_{\mathbb P}(-1) \longrightarrow 
q^{\ast}{\mathcal A} \longrightarrow {\mathcal A}^{\prime} \longrightarrow 0 
\end{equation*}   
one gets an exact sequence $0 \rightarrow {\mathcal O}_L(-1) 
\rightarrow {\mathcal O}_L^{\oplus \, d+1} \rightarrow {\mathcal A}^{\prime}_L 
\rightarrow 0$ from which one deduces that 
${\mathcal A}^{\prime}_L \simeq \text{T}_L(-1)$. 
Now, from Lemma~\ref{L:diffgrass}, ${\Omega}_{{\mathbb F}/{\mathbb P}} \simeq 
{\mathcal Hom}_{{\mathcal O}_{\mathbb F}}
(q^{\ast}{\mathcal B},{\mathcal A}^{\prime})$ hence 
\begin{equation*}
{\Omega}_{{\mathbb F}/{\mathbb P}}\vert_{q^{-1}(\ell)} \simeq 
{\mathcal Hom}_{{\mathcal O}_L}({\mathcal O}_L^{\oplus \, n-d},
{\mathcal A}^{\prime}_L) \simeq \text{T}_L(-1)^{\oplus \, n-d}\, . 
\qedhere 
\end{equation*}
\end{proof} 

Let us illustrate the way the Standard Construction method works by an 
easy example, of Grauert-M\"{u}lich type (cf. \cite[Cor.~1.5]{ES84}). 

\begin{prop}\label{P:restrconics}
Let $\mathcal E$ be a semistable rank 2 reflexive sheaf on a smooth quadric 
hypersurface $Q = Q_{n-1} \subset {\mathbb P}^n$, $n \geq 3$, with 
${\fam0 det}\, {\mathcal E} \simeq {\mathcal O}_Q(c_1)$, 
$c_1 \in {\mathbb Z}$. Then, for a general smooth conic $C \subset Q$ 
(avoiding, in particular, the singular points of $\mathcal E$), 
${\mathcal E}\vert_C \simeq {\mathcal O}_{{\mathbb P}^1}(c_1)^{\oplus 2}$. 
\end{prop}

\begin{proof}
We may assume that $c_1 = 0$ or $-1$. 
Consider the incidence diagram from the statement of Lemma~\ref{L:relomega} 
for $d = 2$. Let $X := p^{-1}(Q)$ and consider the induced diagram: 
\begin{equation*}
\begin{CD}
X @>{\overline q}>> {\mathbb G}_2({\mathbb P}^n)\\
@V{\overline p}VV \\
Q
\end{CD}
\end{equation*}
Let $U$ be the open subset of 
${\mathbb G} := {\mathbb G}_2({\mathbb P}^n)$ consisting of the points 
$\ell$ for which the corresponding 2-plane $L$ intersects $Q$ transversally, 
along a conic $C$ avoiding the singular points of $\mathcal E$. 
$\overline p$ maps ${\overline q}^{\, -1}(\ell)$ isomorphically onto $C$. 
By semicontinuity, there exists an integer $a \geq 0$ and a non-empty open 
subset $U^{\prime}$ of $U$ such that ${\mathcal E}\vert_C \simeq 
{\mathcal O}_{{\mathbb P}^1}(c_1+a) \oplus {\mathcal O}_{{\mathbb P}^1}(c_1-a)$, 
$\forall \ell \in U^{\prime}$.  
We want to show that $a = 0$. If $a \geq 1$ then the image of the canonical 
morphism ${\overline q}^{\ast}{\overline q}_{\ast}{\overline p}^{\ast}
{\mathcal E}\vert_{{\overline q}^{\, -1}(U^{\prime})} \rightarrow 
{\overline p}^{\ast}{\mathcal E}\vert_{{\overline q}^{\, -1}(U^{\prime})}$ is 
a line subbundle ${\mathcal E}^{\prime}$ of 
${\overline p}^{\ast}{\mathcal E}\vert_{{\overline q}^{\, -1}(U^{\prime})}$ 
such that ${\mathcal E}^{\prime}\vert_{{\overline q}^{\, -1}(\ell)} \simeq 
{\mathcal O}_{{\mathbb P}^1}(c_1+a)$, $\forall \ell \in U^{\prime}$. Applying 
Prop.~\ref{P:standconstr} to the restriction of 
${\overline p} : {\overline q}^{\, -1}
(U^{\prime}) \rightarrow Q\setminus \text{Sing}\, {\mathcal E}$ and to the 
exact sequence; 
\begin{equation*}
0 \longrightarrow {\mathcal E}^{\prime} \longrightarrow 
{\overline p}^{\ast}{\mathcal E}\vert_{{\overline q}^{\, -1}(U^{\prime})}  
\longrightarrow {\mathcal E}^{\prime \prime} \longrightarrow 0
\end{equation*}
(where ${\mathcal E}^{\prime \prime}$ is the cokernel of the left morphism) and 
taking into account the semistability of $\mathcal E$ one deduces that the 
second fundamental form ${\mathcal E}^{\prime} \rightarrow 
({\Omega}_{X/Q}\vert_{{\overline q}^{\, -1}(U^{\prime})}) \otimes 
{\mathcal E}^{\prime \prime}$ has generically rank $\geq 1$, hence its 
restriction to a general fibre ${\overline q}^{\, -1}(\ell)$, $\ell \in  
U^{\prime}$, must have generically rank $\geq 1$. 
Using Lemma~\ref{L:relomega}, one deduces the existence of 
a non-zero morphism:
\begin{equation*}
{\mathcal O}_{{\mathbb P}^1}(c_1+a) \longrightarrow 
(\text{T}_L(-1)^{\oplus \, n-2}\vert_C)\otimes 
{\mathcal O}_{{\mathbb P}^1}(c_1-a)\, .
\end{equation*}
Since $ \text{T}_L(-1)\vert_C \simeq 
{\mathcal O}_{{\mathbb P}^1}(1)^{\oplus 2}$, one derives a contradiction.      
\end{proof}  

\begin{lem}\label{L:sigma}
Consider the incidence diagram from the statement of Lemma~\ref{L:relomega} 
with $d = n-1$, i.e., with ${\mathbb G}_d({\mathbb P}^n) = 
{\mathbb P}^{n\vee}$. Let $\Sigma \subset {\mathbb P}^{n\vee}$ be a closed 
reduced and irreducible subscheme, of dimension $m \geq 2$, let 
$X := q^{-1}(\Sigma)$ and consider the induced incidence diagram: 
\begin{equation*}
\begin{CD}
X @>{\overline q}>> \Sigma \\
@V{\overline p}VV \\
{\mathbb P}^n
\end{CD}
\end{equation*} 

\emph{(i)} Let $h$ be a nonsingular point of $\Sigma$ and let 
${\fam0 T}_h\Sigma \subset {\mathbb P}^{n\vee}$ 
be the tangent linear space of $\Sigma$ at $h$. 
Let $H\subset {\mathbb P}^n$ be the hyperplane corresponding to $h$. 
One has ${\fam0 T}_h\Sigma = L^{\vee}$ for some linear subspace $L$ of 
$H$ with ${\fam0 codim}(L,H) = m$. Then one has an exact sequence: 
\begin{equation*}
0 \longrightarrow {\mathcal O}_H(-1) \longrightarrow {\mathcal O}_H^{\oplus m} 
\longrightarrow {\Omega}_{X/{\mathbb P}}\vert_{{\overline q}^{\, -1}(h)}  
\longrightarrow 0
\end{equation*}
where the left morphism is the dual of an epimorphism 
${\mathcal O}_H^{\oplus m} \rightarrow {\mathcal I}_{L,H}(1)$. 

\emph{(ii)} If $x \in {\mathbb P}^n$ then the fibre 
${\overline p}^{\, -1}(x)$ has pure dimension $m - 1$, 
except in the case where $\Sigma = K^{\vee}$ for 
some linear subspace $K$ of ${\mathbb P}^n$ of codimension $m + 1$ and 
$x \in K$. 
 
\emph{(iii)} The set of points $x \in {\mathbb P}^n$ for which the fibre 
${\overline p}^{\, -1}(x)$ is not irreducible and generically reduced is a 
closed subset of ${\mathbb P}^n$, of codimension $\geq m-1$. 
\end{lem}

\begin{proof}
If $x \in {\mathbb P}^n$ and if $x^{\vee}$ is the hyperplane of 
${\mathbb P}^{n\vee}$ consisting of the points $h$ for which $H \ni x$, 
then $\overline q$ maps isomorphically the fibre ${\overline p}^{\, -1}(x)$ 
onto the scheme $x^{\vee} \cap \Sigma$. (ii) is, now, clear and (iii) follows 
from a recent result of O. Benoist \cite[Thm.~0.5(i)]{Be11} (see 
Remark~\ref{R:proofbenoist} below). 

(i) Let ${\mathbb F} := {\mathbb F}_{0,n-1}({\mathbb P}^n)$. One has an exact 
sequence: 
\begin{equation*}
{\overline q}^{\, \ast}({\mathcal I}_{\Sigma}/{\mathcal I}_{\Sigma}^2) 
\longrightarrow {\Omega}_{{\mathbb F}/{\mathbb P}}\vert_X 
\longrightarrow {\Omega}_{X/{\mathbb P}} \longrightarrow 0\, .
\end{equation*}
Restricting this exact sequence to ${\overline q}^{\, -1}(h)$ and taking into 
account Lemma~\ref{L:relomega}, one gets an exact sequence: 
\begin{equation*}
{\mathcal O}_H^{\oplus \, n-m} \longrightarrow \text{T}_H(-1) 
\longrightarrow  {\Omega}_{X/{\mathbb P}}\vert_{{\overline q}^{\, -1}(h)}   
\longrightarrow 0\, .
\end{equation*}
But, if $x \in H \setminus L$, $x^{\vee}$ intersects transversally 
$\Sigma$ at $h$. Identifying ${\overline q}^{\, -1}(h)$ with $H$ via 
$\overline p$, one deduces that 
${\Omega}_{X/{\mathbb P}}\vert_{{\overline q}^{\, -1}(h)}$ is locally free 
of rank $m-1$ over $H\setminus L$. One derives that the morphism 
${\mathcal O}_H^{\oplus \, n-m} \rightarrow \text{T}_H(-1)$ is a 
monomorphism. Using Euler's exact sequence on $H$: 
\begin{equation*}
0 \longrightarrow {\mathcal O}_H(-1) \longrightarrow 
{\mathcal O}_H^{\oplus \, n} \longrightarrow \text{T}_H(-1) 
\longrightarrow 0
\end{equation*}
one gets an exact sequence as in the statement. The left morphism in this 
exact sequence degenerates only along $L$, hence it must be the dual of an 
epimorphism as in the statement.   
\end{proof}

\begin{rmk}\label{R:proofbenoist}
The result of O. Benoist quoted in the proof of Lemma~\ref{L:sigma} asserts 
that if $Z$ is a reduced and irreducible closed subscheme of ${\mathbb P}^n$, 
of dimension $m \geq 2$, then the set of points $h \in {\mathbb P}^{n\vee}$ for 
which the scheme $H \cap Z$ is not generically reduced, irreducible, of 
dimension $m-1$ is a closed subset of ${\mathbb P}^{n\vee}$, of codimension 
$\geq m-1$. We shall, actually, use Lemma~\ref{L:sigma} only in the case 
where $\Sigma$ is a hypersurface in ${\mathbb P}^{n\vee}$, hence we need 
the result of Benoist only in the case where $Z$ is a hypersurface in 
${\mathbb P}^n$. In this particular case, the arguments used by Benoist 
become substantially simpler. 

Indeed, if $Z$ is a hypersurface of degree $d$, \cite[Prop.~1.1]{Be11} 
can be replaced by the following statement: let $\mathcal E$ be a locally 
free sheaf of rank $n$ on a scheme $T$ and let $\pi : {\mathbb P}(\mathcal E) 
\rightarrow T$ be the associated projective bundle. Let $X$ be an effective 
relative Cartier divisor on ${\mathbb P}(\mathcal E)/T$ such that, 
$\forall t \in T$, $X_t$ is a hypersurface of degree $d$ in 
${\mathbb P}({\mathcal E}(t))$. Then the set of the points $t \in T$ such that 
$X_t$ is reduced and irreducible is an open subset of $T$. 

The proof of this statement follows from the following easy fact: consider 
the polynomial ring  
$S := {\mathbb C}[X_0,\ldots ,X_{n-1}]$. Then the set of the points 
$[f] \in {\mathbb P}(S_d)$ such that the polynomial $f$ is irreducible is 
open because its complement is the union of the images of the morphisms 
${\mathbb P}(S_e) \times {\mathbb P}(S_{d-e}) \rightarrow {\mathbb P}(S_d)$, 
$1\leq e \leq d/2$. 

Secondly, continuing to assume that $Z$ is a hypersurface, one reduces the 
proof of the result of Benoist, as in \cite[Prop.~2.5]{Be11}, to the case 
where $Z$ is the cone over a reduced and irreducible plane curve, with 
vertex a linear subspace $L$ of ${\mathbb P}^n$ of dimension $n-3$. 

Finally, in the case where $Z$ is a cone as above, the set from the statement 
of the result of Benoist consists of the hyperplanes containing $L$.
\end{rmk}

\section{The spinor bundle and lines on a quadric threefold}
\label{S:lines}

\subsection{The variety of lines on a quadric threefold}
\label{SS:varlines} 
We identify $Q = Q_3 \subset {\mathbb P}^4$ with a smooth hyperplane section 
${\mathbb P}^4 \cap {\mathbb G}_2({\mathbb C}^4)$ of the Pl\"{u}cker 
embedding into ${\mathbb P}^5$ of the Grassmannian of lines in 
${\mathbb P}^3$, ${\mathbb G}_1({\mathbb P}^3) = 
{\mathbb G}_2({\mathbb C}^4)$. More precisely, let $U := {\mathbb C}^4$ and 
consider the Pl\"{u}cker embedding ${\mathbb G}_2(U) \hookrightarrow 
{\mathbb P}(\overset{2}{\bigwedge}U) = {\mathbb P}^5$. The image of this 
embedding is the quadric 4-fold of ${\mathbb P}^5$ of equation $w \wedge w = 
0$, $w \in \overset{2}{\bigwedge}U$. 

The restriction to ${\mathbb G}_2(U)$ of the universal skew-symmetric morphism 
on ${\mathbb P}^5$: 
\begin{equation*}
U^{\ast}\otimes_{\mathbb C}{\mathcal O}_{{\mathbb P}^5}(-1) 
\longrightarrow U^{\ast}\otimes_{\mathbb C}\overset{2}{\textstyle \bigwedge}U 
\otimes_{\mathbb C}{\mathcal O}_{{\mathbb P}^5} 
\longrightarrow U\otimes_{\mathbb C}
{\mathcal O}_{{\mathbb P}^5} 
\end{equation*}
is a morphism of constant rank 2, whose image is the universal subbundle 
$\mathcal A$ of $U\otimes_{\mathbb C}{\mathcal O}_{\mathbb G}$ and whose 
cokernel is the universal quotient bundle $\mathcal B$. One has an 
isomorphism:
\begin{equation*}
U^{\ast} = \text{Hom}_{{\mathcal O}_{\mathbb G}}(U\otimes_{\mathbb C}
{\mathcal O}_{\mathbb G}, {\mathcal O}_{\mathbb G}) \overset{\sim}
{\longrightarrow} \text{Hom}_{{\mathcal O}_{\mathbb G}}({\mathcal A}, 
{\mathcal O}_{\mathbb G}) 
\end{equation*}     
and the image of the morphism ${\mathcal A} \rightarrow 
{\mathcal O}_{\mathbb G}$ corresponding to $0 \neq \lambda \in U^{\ast}$ is the 
ideal sheaf of the 2-plane ${\mathbb P}(\overset{2}{\bigwedge}
(\text{Ker}\, \lambda)) \subset {\mathbb G}_2(U)$. Let 
${\mathbb F}_{0,1}({\mathbb P}(U)) = {\mathbb F}_{1,2}(U)$ be the flag variety 
``point $\in$ line $\subset {\mathbb P}^3$'' and consider the incidence 
diagram: 
\begin{equation*}
\begin{CD}
{\mathbb F}_{1,2}(U) @>{\widetilde q}>> {\mathbb P}(U)\\ 
@V{\widetilde p}VV \\
{\mathbb G}_2(U) 
\end{CD}
\end{equation*} 
$\widetilde p$ maps isomorphically the fibre ${\widetilde q}^{-1}([u])$ onto 
${\mathbb P}(u \wedge U) \subset {\mathbb G}_2(U)$. Considering the 
tautological coKoszul sequence and the Euler sequence on ${\mathbb P}(U)$: 
\begin{gather*}
0 \rightarrow {\mathcal O}_{\mathbb P} \rightarrow U\otimes 
{\mathcal O}_{\mathbb P}(1) \rightarrow 
\overset{2}{\textstyle{\bigwedge}}U\otimes {\mathcal O}_{\mathbb P}(2)  
\rightarrow 
\overset{3}{\textstyle{\bigwedge}}U\otimes {\mathcal O}_{\mathbb P}(3)  
\rightarrow 
\overset{4}{\textstyle{\bigwedge}}U\otimes {\mathcal O}_{\mathbb P}(4)  
\rightarrow 0\\
0 \rightarrow {\mathcal O}_{\mathbb P}(-1) \rightarrow U\otimes 
{\mathcal O}_{\mathbb P} \rightarrow \text{T}_{\mathbb P}(-1) \rightarrow 0 
\end{gather*}
one sees that 
${\mathbb F}_{1,2}(U) \simeq {\mathbb P}(\text{T}_{\mathbb P}(-2))$ 
over ${\mathbb P}(U)$ such that 
${\mathcal O}_{{\mathbb P}(\text{T}(-2))}(-1) \simeq 
{\widetilde p}^{\ast}{\mathcal O}_{\mathbb G}(-1)$. Moreover, 
${\mathbb F}_{1,2}(U) \simeq {\mathbb P}(\mathcal A)$ over ${\mathbb G}_2(U)$ 
such that ${\mathcal O}_{{\mathbb P}(\mathcal A)}(-1) \simeq 
{\widetilde q}^{\ast}{\mathcal O}_{{\mathbb P}(U)}(-1)$. 

Now, a linear form on ${\mathbb P}^5 = {\mathbb P}(\overset{2}{\bigwedge}U)$ 
can be identified with a skew-symmetric form $\omega : \overset{2}{\bigwedge}U 
\rightarrow {\mathbb C}$. The hyperplane ${\mathbb P}^4 \simeq K_{\omega} := 
{\mathbb P}(\text{Ker}\, \omega) \subset {\mathbb P}(\overset{2}{\bigwedge}U)$ 
intersects transversally ${\mathbb G}_2(U)$ if and only if $\omega$ is 
non-degenerate. Assume that this is the case and put $Q = Q_{\omega} := 
K_{\omega} \cap {\mathbb G}_2(U)$ and ${\mathcal S} = {\mathcal S}_{\omega} := 
{\mathcal A}\, \vert \, Q_{\omega}$. ${\mathcal S}_{\omega}$ is the so-called 
{\it spinor bundle} on $Q_{\omega}$. Let ${\mathbb F}_{0,1}(Q_{\omega}) := 
{\widetilde p}^{-1}(Q_{\omega})$ and consider the incidence diagram:
\begin{equation*}
\begin{CD}
{\mathbb F}_{0,1}(Q_{\omega}) @>q>> {\mathbb P}(U)\\
@VpVV \\ 
Q_{\omega}
\end{CD} 
\end{equation*}
Of course, ${\mathbb F}_{0,1}(Q_{\omega}) \simeq 
{\mathbb P}({\mathcal S}_{\omega})$ over $Q_{\omega}$. On the other hand, $p$ 
maps isomorphically the fibre $q^{-1}([u])$ onto the line $K_{\omega} \cap 
{\mathbb P}(u\wedge U) \subset Q_{\omega}$. One can associate to $\omega$ a 
so-called {\it null correlation bundle} $N_{\omega}$ over ${\mathbb P}^3 = 
{\mathbb P}(U)$ defined as the cokernel of the composite morphism: 
\begin{equation*}
\begin{CD}
{\mathcal O}_{{\mathbb P}(U)}(-1) @>{\omega \otimes \text{id}}>>   
{\overset{2}{\textstyle \bigwedge} U^{\ast}\otimes 
{\mathcal O}_{{\mathbb P}(U)}(-1)} @>>> 
{{\Omega}_{{\mathbb P}(U)}(1)\  .}
\end{CD}
\end{equation*}
It follows that the kernel of the composite morphism:
\begin{equation*}
\begin{CD}
\text{T}_{{\mathbb P}(U)}(-2) @>>> {\overset{2}{\textstyle \bigwedge} U\otimes 
{\mathcal O}_{{\mathbb P}(U)}} @>{\omega \otimes \text{id}}>>  
{\mathcal O}_{{\mathbb P}(U)}
\end{CD}
\end{equation*}
is $N_{\omega}^{\ast}(-1)$, hence ${\mathbb F}_{0,1}(Q_{\omega}) \simeq 
{\mathbb P}(N_{\omega}^{\ast}(-1))$ over ${\mathbb P}(U)$. 

\subsection{The spinor bundle}\label{SS:spinor}
Keeping the notation from par.~\ref{SS:varlines}, 
the restriction to $Q_{\omega}$ of 
the isomorphism $\omega \otimes \text{id} : U\otimes_{\mathbb C}
{\mathcal O}_{\mathbb G} \overset{\sim}{\longrightarrow} 
U^{\ast}\otimes_{\mathbb C}{\mathcal O}_{\mathbb G}$ induces an isomorphism 
${\mathcal A}\vert_{Q_{\omega}} \overset{\sim}{\rightarrow} 
{\mathcal B}^{\ast}\vert_{Q_{\omega}}$ hence one gets an exact sequence: 
\begin{equation}
\label{tautologiconq}
0 \longrightarrow {\mathcal S}_{\omega} \longrightarrow U\otimes_{\mathbb C} 
{\mathcal O}_Q \longrightarrow {\mathcal S}_{\omega}^{\ast} \longrightarrow 0
\ .
\end{equation}
Moreover, since $\text{det}\, {\mathcal A} \simeq {\mathcal O}_{\mathbb G}(-1)$ 
one has $\text{det}\, {\mathcal S}_{\omega} \simeq {\mathcal O}_Q(-1)$ hence 
${\mathcal S}_{\omega}^{\ast} \simeq {\mathcal S}_{\omega}(1)$. 

The morphism ${\mathcal A} \rightarrow {\mathcal O}_{\mathbb G}$ defined by a 
non-zero $\lambda \in U^{\ast}$ restricts to a morphism ${\mathcal S}_{\omega} 
\rightarrow {\mathcal O}_Q$ whose image is the ideal sheaf of the line 
$L := K_{\omega} \cap {\mathbb P}(\overset{2}{\bigwedge}
(\text{Ker}\, \lambda)) \subset Q_{\omega}$. One derives an exact sequence: 
\begin{equation}
\label{spinorsections}
0 \longrightarrow {\mathcal O}_Q(-1) \longrightarrow {\mathcal S}_{\omega} 
\longrightarrow {\mathcal I}_L \longrightarrow 0
\end{equation}
from which one can, of course, compute the cohomology of 
${\mathcal S}_{\omega}$ and of its twists. 

Finally, let $H$ be a hyperplane in $K_{\omega} \simeq {\mathbb P}^4$ which 
intersects $Q_{\omega}$ transversally. In this case, $H \cap Q_{\omega}$ is a 
smooth quadric in $H \simeq {\mathbb P}^3$. Let $L$ be a line belonging to 
the first ruling $\vert \, {\mathcal O}_{H \cap Q}(1,0)\, \vert$ of 
$H \cap Q_{\omega}$. From (\ref{spinorsections}) one gets an epimorphism 
${\mathcal S}_{\omega}\vert_{H \cap Q} \rightarrow 
{\mathcal O}_{H \cap Q}(-L) \simeq {\mathcal O}_{H \cap Q}(-1,0)$. Since 
$\text{det}\, {\mathcal S}_{\omega} \simeq {\mathcal O}_Q(-1)$, one deduces an 
exact sequence: 
\begin{equation*}
0 \longrightarrow {\mathcal O}_{H \cap Q}(0,-1) \longrightarrow 
{\mathcal S}_{\omega}\vert_{H \cap Q} \longrightarrow 
{\mathcal O}_{H \cap Q}(-1,0) \longrightarrow 0\  .
\end{equation*} 
This exact sequence splits because $\text{H}^1({\mathcal O}_{H \cap Q}(1,-1)) = 
0$. Consequently: 
\begin{equation}
\label{spinorresthyp}
{\mathcal S}_{\omega}\vert_{H \cap Q} \simeq {\mathcal O}_{H \cap Q}(-1,0) 
\oplus {\mathcal O}_{H \cap Q}(0,-1)\  .
\end{equation}  

\subsection{Lines on hyperplane sections of a quadric threefold}
\label{SS:linesonhyp}
We described in par.~\ref{SS:varlines} the family $q : {\mathbb F}_{0,1}(Q) 
\rightarrow {\mathbb P}(U) = {\mathbb P}^3$ of lines on $Q = Q_{\omega} 
\subset K_{\omega} = {\mathbb P}^4$. For a point $\ell = [u] \in 
{\mathbb P}(U)$, let us denote by $L = {\mathbb P}(u \wedge u^{\perp})$ the 
corresponding line on $Q$, where $u^{\perp} \subset U$ is the orthogonal of $u$ 
with respect to $\omega$. We want to describe the flag variety 
${\mathbb F}_{1,2}(Q) \subset {\mathbb P}(U) \times {\mathbb P}^{4\vee}$ 
consisting of the pairs $(\ell,h)$ with $L \subset H \cap Q$. 

One can identify ${\mathbb P}^{4\vee}$ with 
${\mathbb P}(\text{H}^0({\mathcal O}_Q(1)))$. We saw that 
${\mathbb F}_{0,1}(Q) \simeq {\mathbb P}(N_{\omega}^{\ast}(-1))$ over 
${\mathbb P}(U)$ such that ${\mathcal O}_{{\mathbb P}(N^{\ast}(-1))}(-1) \simeq 
p^{\ast}{\mathcal O}_Q(-1)$. One deduces that $q_{\ast}p^{\ast}{\mathcal O}_Q(1) 
\simeq N_{\omega}(1)$, hence $\text{H}^0({\mathcal O}_Q(1)) \simeq 
\text{H}^0(N_{\omega}(1))$. Let $0 \neq h \in \text{H}^0({\mathcal O}_Q(1))$. 
It corresponds to a section $s \in \text{H}^0(N_{\omega}(1))$. Let $\ell = [u] 
\in {\mathbb P}(U)$. $p : {\mathbb F}_{0,1}(Q) \rightarrow Q$ maps 
$q^{-1}(\ell)$ isomorphically onto the corresponding line $L \subset Q$. It 
follows that: 
\begin{equation*}
h\  \text{vanishes~on}\  L\  \Leftrightarrow \  p^{\ast}h\  \text{vanishes~on} 
\  q^{-1}(\ell)\  \Leftrightarrow \  s = q_{\ast}p^{\ast}h\  \text{vanishes~in} 
\  \ell\, .  
\end{equation*}
Consequently, ${\mathbb F}_{1,2}(Q) \subset {\mathbb P}(U) \times 
{\mathbb P}^{4\vee}$ can be identified with $Z \subset {\mathbb P}(U) \times 
{\mathbb P}(\text{H}^0(N_{\omega}(1)))$ consisting of the pairs $(\ell,[s])$ 
with $s(\ell) = 0$. Let $M_{\omega}$ be defined by the exact sequence: 
\begin{equation*}
0 \longrightarrow M_{\omega} \longrightarrow  
\text{H}^0(N_{\omega}(1))\otimes_{\mathbb C}{\mathcal O}_{{\mathbb P}(U)} 
\overset{\text{ev}}{\longrightarrow} N_{\omega}(1) 
\longrightarrow 0\  .
\end{equation*}
Consider the two projections: 
\begin{equation*}
\begin{CD}
{\mathbb F}_{1,2}(Q) @>q_1>> {\mathbb P}^{4\vee}\\
@Vp_1VV \\
{\mathbb P}(U)
\end{CD}
\end{equation*}
From the above discussion, ${\mathbb F}_{1,2}(Q) \simeq 
{\mathbb P}(M_{\omega})$ over ${\mathbb P}(U)$ and if $h \in 
{\mathbb P}^{4\vee}$ corresponds to $[s] \in {\mathbb P}(\text{H}^0
(N_{\omega}(1)))$ then $p_1$ maps isomorphically $q_1^{-1}(h)$ onto the scheme 
of zeroes $Z(s)$ of $s$. As it is well-known, if $h \in {\mathbb P}^{4\vee} 
\setminus Q^{\vee}$ then $Z(s)$ is the union of two disjoint lines (which 
correspond to the two rulings of $H \cap Q \simeq {\mathbb P}^1 \times 
{\mathbb P}^1$), and if $h \in Q^{\vee}$ then $Z(s)$ is a double structure on 
a line (in this case, $H \cap Q$ is a quadratic cone).

\section{Some auxiliary results on the quadric surface}
\label{S:auxq2} 

Let $Q_2 \subset {\mathbb P}^3$ be a smooth quadric, $Q_2 \simeq {\mathbb P}^1 
\times {\mathbb P}^1$, and let $p_1,p_2 : Q_2 \rightarrow {\mathbb P}^1$ be 
the projections. For $a,b \in {\mathbb Z}$, we put ${\mathcal O}_{Q_2}(a,b) 
:= p_1^{\ast}{\mathcal O}_{{\mathbb P}^1}(a) \otimes 
p_2^{\ast}{\mathcal O}_{{\mathbb P}^1}(b)$. Throughout this section, $F$ will 
denote a rank 2 vector bundle on $Q_2$, with $\text{det}\, F \simeq 
{\mathcal O}_{Q_2}(c_1,c_1)$, $c_1 = 0$ or $-1$, and $c_2(F) = c_2 \in 
{\mathbb Z}$. We shall denote by $\overline {\mathcal S}$ the direct sum 
${\mathcal O}_{Q_2}(-1,0) \oplus {\mathcal O}_{Q_2}(0,-1)$ (if one views $Q_2$ 
as a hyperplane section of a quadric threefold $Q_3 \subset {\mathbb P}^4$ 
then, according to (\ref{spinorresthyp}), ${\overline {\mathcal S}} \simeq 
{\mathcal S}\vert_{Q_2}$). 

\begin{dfn}\label{D:gm}
We say that a vector bundle $F$ as above (i.e., $\text{rk}(F) = 2$ and 
$\text{det}\, F \simeq {\mathcal O}_{Q_2}(c_1,c_1)$, $c_1 = 0$ or $-1$)    
satisfies the {\it Grauert-M\"{u}lich property} if, for any 
general line $L$ on each of the two rulings of $Q_2$, $F\vert_L \simeq 
{\mathcal O}_L \oplus {\mathcal O}_L(c_1)$. 
\end{dfn}

\begin{lem}\label{L:gms}
Assume that $F$ satisfies the Grauert-M\"{u}lich property. Then: 

\emph{(i)} $F$ is stable if and only if ${\fam0 H}^0(F) = 0$ in the case 
$c_1 = 0$, and if and only if ${\fam0 Hom}({\overline {\mathcal S}},F) = 0$ 
in the case $c_1 = -1$. 

\emph{(ii)} If $F$ is semistable but not stable and $c_2 > -c_1$ then 
${\fam 0 h}^0(F) = 1$ in the case $c_1 = 0$, and 
${\fam0 hom}({\overline {\mathcal S}},F) = 1$ in the case $c_1 = -1$.
\end{lem}

\begin{proof}
Assume that $F$ contains a saturated subsheaf (see Def.~\ref{D:satsubsheaf}) 
of the form ${\mathcal O}_{Q_2}(a,b)$. Then one has an exact sequence: 
\begin{equation*}
0 \longrightarrow {\mathcal O}_{Q_2}(a,b) \longrightarrow F \longrightarrow 
{\mathcal I}_Z(c_1-a,c_1-b)\longrightarrow 0
\end{equation*}
where $Z$ is a 0-dimensional subscheme of $Q_2$. Restricting this exact 
sequence to a general line, avoiding $Z$, on each of the two rulings of $Q_2$ 
and using the Grauert-M\"{u}lich property, one derives that $a \leq 0$ and 
$b \leq 0$. By definition, $F$ is (semi)stable if and only if, for each 
saturated subsheaf of $F$ of the form ${\mathcal O}_{Q_2}(a,b)$, one has 
$a+b\  (\leq)< c_1$. If $a \leq 0$ and $b \leq 0$ and $a+b \geq c_1$ then 
$(a,b) = (0,0)$ in the case $c_1 = 0$, and $(a,b) = (-1,0)$ or $(0,-1)$ 
in the case $c_1 = -1$. 

(i) is, now, clear.

(ii) If $F$ is semistable but not stable then, in the case $c_1 = 0$, it can 
be realized as an extension:
\begin{equation*}
0 \longrightarrow {\mathcal O}_{Q_2} \longrightarrow F \longrightarrow 
{\mathcal I}_Z \longrightarrow 0
\end{equation*}
with $\text{deg}\, Z = c_2$, and, in the case $c_1 = -1$, it can be realized 
as an extension of one of the following two types: 
\begin{gather*}
0 \longrightarrow {\mathcal O}_{Q_2}(-1,0) \longrightarrow F \longrightarrow 
{\mathcal I}_Z(0,-1) \longrightarrow 0\\
0 \longrightarrow {\mathcal O}_{Q_2}(0,-1) \longrightarrow F \longrightarrow 
{\mathcal I}_Z(-1,0) \longrightarrow 0
\end{gather*}
with $\text{deg}\, Z = c_2-1$. The condition $c_2 > -c_1$ implies that $Z \neq 
\emptyset$, and assertion (ii) follows.
\end{proof}   

\begin{lem}\label{L:jump}
Assume that $F$ is semistable and let $L \subset Q_2$ be a line. Assume, to 
fix the ideas, that $L$ belongs to the linear system 
$\vert \, {\mathcal O}_{Q_2}(1,0)\, \vert$. 
Let $a$ be the nonnegative integer  
such that $F\vert_L \simeq {\mathcal O}_L(a) \oplus 
{\mathcal O}_L(-a+c_1)$. 

Then $a \leq c_2+c_1$ and if $a = c_2+c_1$ then there exists a $0$-dimensional 
subscheme $Z$ of $L$ such that $F$ can be realized as an extension: 
\begin{equation*}
0 \longrightarrow {\mathcal O}_{Q_2} \longrightarrow F \longrightarrow 
{\mathcal I}_Z \longrightarrow 0
\end{equation*} 
in the case $c_1 = 0$, and as an extension:
\begin{equation*}
0 \longrightarrow {\mathcal O}_{Q_2}(-1,0) \longrightarrow F \longrightarrow 
{\mathcal I}_Z(0,-1) \longrightarrow 0
\end{equation*}
in the case $c_1 = -1$. In particular, $F$ is not stable in this case. 
\end{lem}

\begin{proof}
Assume, firstly, that $c_1 = 0$. In this case, $\text{h}^0(F) \leq 1$, 
unless $F \simeq {\mathcal O}_{Q_2}^{\oplus 2}$. {\it Indeed}, since $F$ is 
semistable, $\text{H}^0(F(b,c)) = 0$ if $b + c < 0$. It follows that the 
zero scheme $Z$ of a non-zero global section $s$ of $F$ has no divisorial 
components, i.e., it is 0-dimensional, hence $F$ can be realized as an 
extension: 
\begin{equation}
\label{extf0}
0 \longrightarrow {\mathcal O}_{Q_2} 
\overset{s}{\longrightarrow} F \longrightarrow 
{\mathcal I}_Z \longrightarrow 0\, .
\end{equation}
In particular, $\text{deg}\, Z = c_2$. One derives that $\text{h}^0(F) \leq 1$ 
unless $Z = \emptyset$ in which case $F \simeq {\mathcal O}_{Q_2}^{\oplus 2}$. 

Assume, now, that $F$ is not trivial. Tensorizing by $F$ the short exact 
sequence:
\begin{equation}
\label{resolol}
0 \longrightarrow {\mathcal O}_{Q_2}(-1,0) \longrightarrow {\mathcal O}_{Q_2} 
\longrightarrow {\mathcal O}_L \longrightarrow 0
\end{equation}
one gets an exact sequence:
\begin{equation*}
0 = \text{H}^0(F(-1,0)) \longrightarrow \text{H}^0(F) \longrightarrow 
\text{H}^0(F\vert_L) \longrightarrow \text{H}^1(F(-1,0))
\end{equation*}
from which one deduces that:
\begin{equation*}
a + 1 = \text{h}^0(F\vert_L) \leq \text{h}^0(F) + \text{h}^1(F(-1,0))\, . 
\end{equation*}
But, by semistability, $\text{H}^0(F(-1,0)) = 0$ and 
\begin{equation*}
\text{H}^2(F(-1,0)) 
\simeq \text{H}^0(F^{\ast}(-1,-2))^{\ast} \simeq \text{H}^0(F(-1,-2))^{\ast} 
= 0\, ,
\end{equation*} 
hence $\text{h}^1(F(-1,0)) = - \chi (F(-1,0)) = c_2$ (one may use the 
exact sequence (\ref{extf0}) tensorized by ${\mathcal O}_{Q_2}(-1,0)$ to 
guess the Riemann-Roch formula in this case). One deduces that $a+1 \leq 
1+c_2$, hence $a \leq c_2$. If $a = c_2$, one must have 
$\text{h}^0(F) = 1$, and $F$ can be realized as an extension (\ref{extf0}). 
One has $s\vert_L \in \text{H}^0(F\vert_L) = 
\text{H}^0({\mathcal O}_L(c_2) \oplus {\mathcal O}_L(-c_2))$, hence 
$\text{deg}(Z\cap L) = c_2$. Since $\text{deg}\, Z = c_2$, it follows that 
$Z \subset L$ as schemes.

Consider, now, the case $c_1 = -1$. $F$ being semistable, 
$\text{H}^0(F(b,c)) = 0$ if $b + c \leq 0$. We assert, firstly, that 
$\text{h}^0(F(1,0)) \leq 1$. {\it Indeed}, as in the case $c_1 = 0$, if 
$F(1,0)$ has a non-zero global section $s$, then the zero scheme $Z$ of $s$ 
is 0-dimensional and $F$ can be realized as an extension:
\begin{equation}
\label{extf-1}
0 \longrightarrow {\mathcal O}_{Q_2}(-1,0) 
\overset{s}{\longrightarrow} F \longrightarrow 
{\mathcal I}_Z(0,-1) \longrightarrow 0\, .
\end{equation}  
In particular, $\text{deg}\, Z = c_2 - 1$. One derives that 
$\text{h}^0(F(1,0)) \leq 1$. 
  
Now, tensorizing by $F(1,0)$ the exact sequence (\ref{resolol}), one gets an 
exact sequence:
\begin{equation*}
0 = \text{H}^0(F) \longrightarrow \text{H}^0(F(1,0)) \longrightarrow 
\text{H}^0(F(1,0)\vert_L) \longrightarrow \text{H}^1(F)\, .
\end{equation*}
Since ${\mathcal O}_{Q_2}(1,0)\vert_L \simeq {\mathcal O}_L$, 
$F(1,0)\vert_L \simeq F\vert_L$. One deduces that:
\begin{equation*}
a+1 = \text{h}^0(F\vert_L) = \text{h}^0(F(1,0)\vert_L) \leq 
\text{h}^0(F(1,0)) + \text{h}^1(F)\, .
\end{equation*}
But, by semistability, $\text{H}^0(F) = 0$ and 
$\text{H}^2(F) \simeq \text{H}^0(F^{\ast}(-2,-2))^{\ast} \simeq 
\text{H}^0(F(-1,-1))^{\ast} = 0$, hence $\text{h}^1(F) = - \chi (F) = 
c_2 - 1$. One deduces that $a + 1 \leq 1 + (c_2 - 1)$, hence $a \leq c_2 - 1$. 
If $a = c_2 - 1$ then $\text{h}^0(F(1,0)) = 1$ and $F$ can be realized as an 
extension (\ref{extf-1}). $s\vert_L \in \text{H}^0(F(1,0)\vert_L) 
= \text{H}^0(F\vert_L) = \text{H}^0({\mathcal O}_L(c_2-1)\oplus 
{\mathcal O}_L(-c_2))$ hence $\text{deg}(Z\cap L) = c_2 - 1$. Since 
$\text{deg}\, Z = c_2 - 1$, it follows that $Z \subset L$ as schemes. 
\end{proof}

\begin{corol}\label{C:jump1}
Assume that $F$ is semistable and that $c_2 > -c_1$. If there is a line 
$L \subset Q_2$ such that $F\vert_L \simeq {\mathcal O}_L(c_2 + c_1) 
\oplus {\mathcal O}_L(-c_2)$ then, for any other line $L^{\prime} \subset 
Q_2$, one has $F\vert_{L^{\prime}} \simeq 
{\mathcal O}_{L^{\prime}}(a^{\prime})\oplus {\mathcal O}_{L^{\prime}}(-a^{\prime} 
+ c_1)$ with $a^{\prime} = 0$ or $1$.
\end{corol}

\begin{proof}
If $L^{\prime}$ and $L$ are on the same ruling of $Q_2$ then $L^{\prime}\cap Z 
= \emptyset$ and if they are on different rulings then $\text{deg}(L^{\prime} 
\cap Z) \leq 1$. If $s$ is the global section of $F$ (resp., $F(1,0)$) 
defining the exact sequence from the last part of the statement of 
Lemma~\ref{L:jump} then the zero scheme of $s\vert_{L^{\prime}}$ is 
$L^{\prime}\cap Z$.  
\end{proof} 

\begin{corol}\label{C:jump2}
Assume that $F$ is semistable. Then there 
exists a line $L \subset Q_2$ such that 
$F\vert_L \simeq {\mathcal O}_L(c_2 + c_1)\oplus {\mathcal O}_L(-c_2)$ 
if and only if ${\fam0 H}^0(F) \neq 0$ and 
${\fam0 hom}({\overline {\mathcal S}},F) \geq 5$ in the case 
$c_1 = 0$, and if and only if ${\fam0 Hom}({\overline {\mathcal S}},F) 
\neq 0$ and ${\fam0 h}^0(F(1,1)) \geq 3$ in the case $c_1 = -1$.
\end{corol} 

\begin{rmk}\label{R:condjump}
Assume that $F$ is semistable. One can show that, in the case $c_1 = 0$, the 
condition ${\fam0 hom}({\overline {\mathcal S}},F) \geq 5$ implies that 
${\fam0 H}^0(F) \neq 0$ and, in the case $c_1 = -1$ and $c_2 > 2$, the 
condition ${\fam0 h}^0(F(1,1)) \geq 3$ implies that 
${\fam0 Hom}({\overline {\mathcal S}},F) \neq 0$. 

On the other hand, let $\mathcal F$ be the rank 2 reflexive sheaf on 
${\mathbb P}^3$ considered in the statement of Lemma~\ref{L:efrestrq2} below. 
Assume that $x \notin Q_2$ and let $F := {\mathcal F}\vert_{Q_2}$. Then 
$\text{det}\, F \simeq {\mathcal O}_{Q_2}(-1,-1)$, $c_2(F) = 2$, 
$\text{h}^0(F(1,1)) = 3$ but, for every line $L \subset Q_2$, 
$F\vert_L \simeq {\mathcal O}_L \oplus {\mathcal O}_L(-1)$.  
\end{rmk} 

\begin{lem}\label{L:gieseker}
If $F$ is Gieseker-Maruyama stable then it is stable. 
\end{lem}

\begin{proof}
$F$ is, at least, semistable. If it were not stable then it would contain a 
saturated subsheaf of the form ${\mathcal O}_{Q_2}(a+c_1,-a)$ for some 
$a \in {\mathbb Z}$ and we would have an exact sequence:
\begin{equation*}
0 \longrightarrow {\mathcal O}_{Q_2}(a+c_1,-a) \longrightarrow F 
\longrightarrow {\mathcal I}_Z(-a,a+c_1) \longrightarrow 0
\end{equation*}
with $Z$ a 0-dimensional subscheme of $Q_2$. If $P(t)$ is the Hilbert 
polynomial of ${\mathcal O}_{Q_2}(a+c_1,-a)$ then the Hilbert polynomial of 
$F$ is $P_F(t) = 2P(t) - \text{deg}\, Z$. The existence of the subsheaf 
${\mathcal O}_{Q_2}(a+c_1,-a)$ would thus contradict the Gieseker-Maruyama 
stability of $F$. 
\end{proof}

\section{Restrictions to nonsingular hyperplane sections}
\label{S:proof}   

\begin{dfn}\label{D:ue}
Let $\mathcal E$ be a rank 2 reflexive sheaf on a smooth quadric threefold 
$Q = Q_3 \subset {\mathbb P}^4$ with $\text{det}\, {\mathcal E} \simeq 
{\mathcal O}_Q(c_1)$, $c_1 = 0$ or $-1$. We put: 
\begin{equation*} 
{\mathcal U}(\mathcal E) := {\mathbb P}^{4\vee} \setminus 
(Q^{\vee} \cup {\textstyle \bigcup}\{x^{\vee}\, \vert \, x\in 
\text{Sing}\, {\mathcal E}\})\, .
\end{equation*}
We also denote by ${\mathcal U}_{\text{gm}}(\mathcal E)$ (resp., 
${\mathcal U}_{\text{ss}}(\mathcal E)$, resp., 
${\mathcal U}_{\text{s}}(\mathcal E)$) the set of those $h \in 
{\mathcal U}(\mathcal E)$ for which ${\mathcal E}\vert_{H\cap Q}$ has 
the Grauert-M\"{u}lich property (see Definition~\ref{D:gm})  
(resp. it is semistable, resp., it is stable). 
\end{dfn}

\begin{lem}\label{L:forgm}
Consider the incidence diagram from par.~\ref{SS:linesonhyp} above and let 
$Z$ be a closed subset of ${\mathbb P}(U) = {\mathbb P}^3$, 
$Z \neq {\mathbb P}^3$. Then the set of points $h \in {\mathbb P}^{4\vee}$ for 
which $Z \cap p_1(q_1^{-1}(h))$ is finite is an open subset of 
${\mathbb P}^{4\vee}$ whose complement has codimension $\geq 2$ in 
${\mathbb P}^{4\vee}$. 
\end{lem}

\begin{proof}
As we saw in par.~\ref{SS:linesonhyp}, the restriction of  
$p_1 : q_1^{-1}(h) \rightarrow {\mathbb P}^3$ 
is a closed immersion, $\forall h \in {\mathbb P}^{4\vee}$. It 
follows that $p_1$ maps $p_1^{-1}(Z) \cap q_1^{-1}(h)$ isomorphically onto 
$Z \cap p_1(q_1^{-1}(h))$. Now, we may assume that $Z$ is an irreducible 
surface in ${\mathbb P}^3$. As ${\mathbb F}_{1,2}(Q)$ is a 
${\mathbb P}^1$-bundle over ${\mathbb P}^3$, $p_1^{-1}(Z)$ is irreducible of 
dimension 3. If $q_1(p_1^{-1}(Z))$ has dimension $\leq 2$ then it is the 
complement of the set from the conclusion of the lemma. If $q_1(p_1^{-1}(Z))$ 
has dimension 3 then, by a well-known theorem of Chevalley (see, for example, 
\cite[I,~\S~8,~Cor.~3]{Mu99}), 
the set $W$ of the points $y\in p_1^{-1}(Z)$ which are isolated in the fibre 
$p_1^{-1}(Z) \cap q_1^{-1}(q_1(y))$ is open in $p_1^{-1}(Z)$ (and non-empty). 
In this case, the complement of the set from the conclusion of the lemma is 
$q_1(p_1^{-1}(Z) \setminus W)$. 
\end{proof}

\begin{lem}\label{L:forss}
Consider the incidence diagram: 
\begin{equation*}
\begin{CD}
{\mathbb F}_{2,3}({\mathbb P}^4) @>g>> {\mathbb P}^{4\vee}\\
@VfVV \\
{\mathbb G}_2({\mathbb P}^4)
\end{CD}
\end{equation*}
and let $Z$ be a closed subset of ${\mathbb G}_2({\mathbb P}^4)$, 
$Z \neq {\mathbb G}_2({\mathbb P}^4)$. Then the set of points $h \in 
{\mathbb P}^{4\vee}$ for which $f(g^{-1}(h)) \subseteq Z$ is a closed subset 
of ${\mathbb P}^{4\vee}$ of codimension $\geq 2$. 
\end{lem}

\begin{proof}
We have $f(g^{-1}(h)) \subseteq Z$ if and only if $g^{-1}(h) \subseteq 
f^{-1}(Z)$. We may assume that $Z$ is an irreducible hypersurface in 
${\mathbb G}_2({\mathbb P}^4)$, hence that it has dimension 5. Since 
${\mathbb F}_{2,3}({\mathbb P}^4)$ is a ${\mathbb P}^1$-bundle over 
${\mathbb G}_2({\mathbb P}^4)$, $f^{-1}(Z)$ is irreducible of dimension 6. 
For $h \in {\mathbb P}^{4\vee}$, $f(g^{-1}(h))$ is a 3-plane in the 
Pl\"{u}cker embedding of ${\mathbb G}_2({\mathbb P}^4)$ in ${\mathbb P}^9$. 
One deduces that a general fibre of the restriction of $g : f^{-1}(Z) 
\rightarrow {\mathbb P}^{4\vee}$ has dimension 2. The set $T$ of the points 
$y \in f^{-1}(Z)$ for which $\text{dim}_y(f^{-1}(Z) \cap g^{-1}(g(y))) \geq 3$ 
is a closed subset of $f^{-1}(Z)$, not equal to $f^{-1}(Z)$. The set from 
the conclusion of the lemma is exactly $f(T)$, which has dimension 
$\leq 5 - 3 = 2$.   
\end{proof}

\begin{prop}\label{P:gmss}
Let $\mathcal E$ be a semistable rank $2$ reflexive sheaf on a quadric 
threefold $Q = Q_3 \subset {\mathbb P}^4$ with ${\fam0 det}\, {\mathcal E} 
\simeq {\mathcal O}_Q(c_1)$, $c_1 = 0$ or $-1$. Then:  

\emph{(i)} ${\mathcal U}_{\fam0 gm}(\mathcal E)$ and 
${\mathcal U}_{\fam0 ss}(\mathcal E)$ are open subsets of 
${\mathcal U}(\mathcal E)$ and their complements in 
${\mathcal U}(\mathcal E)$ have codimension $\geq 2$. 

\emph{(ii)} ${\mathcal U}_{\fam0 s}(\mathcal E) \cap 
{\mathcal U}_{\fam0 gm}(\mathcal E)$ is an open subset of 
${\mathcal U}(\mathcal E)$.
\end{prop}

\begin{proof}
(i) Using the notation from par.~\ref{SS:linesonhyp}, the set of points 
$\ell \in {\mathbb P}(U) = {\mathbb P}^3$ for which the corresponding line 
$L \subset Q$ passes through a singular point of $\mathcal E$ is a union of 
finitely many lines in ${\mathbb P}^3$. Let $W$ be the complement of this 
union of lines. An argument of semicontinuity shows that the set of points 
$\ell \in W$ for which ${\mathcal E}\vert_L \simeq 
{\mathcal O}_L(a) \oplus {\mathcal O}_L(-a+c_1)$ with $a > 0$ 
is a closed subset of $W$ and it is not equal to $W$ by 
\cite[Prop.~1.3]{ES84}. One deduces, now, from Lemma~\ref{L:forgm}, that 
${\mathcal U}_{\text{gm}}(\mathcal E)$ is an open subset of 
${\mathcal U}(\mathcal E)$ and that its complement has codimension $\geq 2$. 

Secondly, if $h \in {\mathcal U}(\mathcal E)$, then 
${\mathcal E}\vert_{H\cap Q}$ 
is semistable if and only if $H \cap Q$ contains a smooth conic 
$C$ such that ${\mathcal E}\vert_C \simeq 
{\mathcal O}_{{\mathbb P}^1}(c_1)^{\oplus 2}$ (by the case $n = 3$ of 
Prop.~\ref{P:restrconics}). The assertion about 
${\mathcal U}_{\text{ss}}(\mathcal E)$ follows, now, from the case $n = 4$ of 
Prop.~\ref{P:restrconics} and from Lemma~\ref{L:forss}. 

(ii) follows from Lemma~\ref{L:gms}(i).      
\end{proof}

\begin{rmk}\label{R:useopen}
${\mathcal U}_{\text{s}}({\mathcal E})$ is, actually, itself an open subset of 
${\mathcal U}({\mathcal E})$. Indeed, this follows from 
Lemma~\ref{L:gieseker} and from the openess of the Gieseker-Maruyama 
stability, a result due to Maruyama~\cite[Thm.~2.8]{Ma76}.
\end{rmk}

\begin{prop}\label{P:nounstablehyp}
Let $E$ be a stable rank $2$ vector bundle on $Q = Q_3 \subset {\mathbb P}^4$ 
with $c_1(E) = c_1 = 0$ or $-1$ and with $c_2 > -c_1$. Let $\Sigma \subset 
{\mathbb P}^{4\vee}$ be an irreducible hypersurface, $\Sigma \neq Q^{\vee}$ 
and $\Sigma \neq x^{\vee}$, $\forall x \in Q$, such that 
${\mathcal U}_{\fam0 s}(E) \cap {\mathcal U}_{\fam0 gm}(E) \cap \Sigma = 
\emptyset$. Let $H_0 \subset {\mathbb P}^4$ be a hyperplane intersecting $Q$ 
transversally. 

If $E\vert_{H_0\cap Q}$ is unstable (i.e., not semistable) then 
$\Sigma = (H_0\cap Q)^{\vee}$. 
\end{prop}

\begin{proof}
Let ${\mathcal O}_{H_0 \cap Q}(a,b) \subset E\vert_{H_0 \cap Q}$ be a 
maximal destabilizing subsheaf. One has $a+b > c_1$ and an exact sequance: 
\begin{equation*}
0 \longrightarrow {\mathcal O}_{H_0 \cap Q}(a,b) \longrightarrow  
E\vert_{H_0 \cap Q} \longrightarrow 
{\mathcal I}_{Z,H_0 \cap Q}(c_1-a,c_1-b) \longrightarrow 0 
\end{equation*}
where $Z$ is a 0-dimensional subscheme of $H_0 \cap Q$. Let 
${\mathcal E}^{\prime}$ be the rank 2 reflexive sheaf on $Q$ defined by the 
exact sequence:
\begin{equation*}
0 \longrightarrow {\mathcal E}^{\prime}(c_1) \longrightarrow E \longrightarrow 
{\mathcal I}_{Z,H_0 \cap Q}(c_1-a,c_1-b) \longrightarrow 0\, 
\end{equation*}
One has $c_1^{\prime} := c_1({\mathcal E}^{\prime}) = -1$ if $c_1 =0$,  
$c_1^{\prime} = 0$ if $c_1 = -1$, $\text{Sing}\, {\mathcal E}^{\prime} = Z$, and 
$E$ stable implies that $\text{H}^0(E) = 0$, hence 
$\text{H}^0({\mathcal E}^{\prime}(c_1)) = 0$ hence ${\mathcal E}^{\prime}$ is 
semistable. 

We will show that if one assumes that $\Sigma \neq (H_0 \cap Q)^{\vee}$ then 
one gets a contradiction. Indeed, by Prop.~\ref{P:gmss}(i), 
${\mathcal U}_{\text{ss}}({\mathcal E}^{\prime}) \cap \Sigma \neq \emptyset$. 
Recall that, by definiton, ${\mathcal U}_{\text{ss}}({\mathcal E}^{\prime})$ 
is an open subset of ${\mathcal U}({\mathcal E}^{\prime}) = 
{\mathbb P}^{4\vee} \setminus (Q^{\vee} \cup \bigcup \{x^{\vee}\, \vert \, 
x \in Z\})$. If 
$\Sigma \neq (H_0 \cap Q)^{\vee}$, choose a point 
\begin{equation*}
h \in ({\mathcal U}_{\text{ss}}({\mathcal E}^{\prime}) \cap 
{\mathcal U}_{\text{ss}}(E) \cap {\mathcal U}_{\text{gm}}(E) \cap \Sigma) 
\setminus (H_0 \cap Q)^{\vee}\, .
\end{equation*} 
If $H \subset {\mathbb P}^4$ is the corresponding hyperplane, then 
$H \cap H_0 \cap Q$ is a smooth conic $C \simeq {\mathbb P}^1$, avoiding $Z$. 
Restricting to $H \cap Q$ the exact sequence defining ${\mathcal E}^{\prime}$, 
one gets an exact sequence: 
\begin{equation*}
0 \longrightarrow {\mathcal E}^{\prime}(c_1)\vert_{H \cap Q}  
\longrightarrow 
E\vert_{H \cap Q} \longrightarrow {\mathcal O}_{{\mathbb P}^1}(2c_1-a-b) 
\longrightarrow 0\, .
\end{equation*}
By the proof of Lemma~\ref{L:gms}(ii), $E\vert_{H \cap Q}$ has a subsheaf 
$\mathcal L$ of the form ${\mathcal O}_{H \cap Q}$ in the case $c_1 = 0$ and of 
the form ${\mathcal O}_{H \cap Q}(-1,0)$ or ${\mathcal O}_{H \cap Q}(0,-1)$, in 
the case $c_1 = -1$. Since ${\mathcal L}\vert_C \simeq 
{\mathcal O}_{{\mathbb P}^1}(c_1)$, it follows that 
$\text{Hom}_{{\mathcal O}_{H \cap Q}}({\mathcal L},
{\mathcal O}_{{\mathbb P}^1}(2c_1-a-b)) = 0$, hence 
${\mathcal L} \subset {\mathcal E}^{\prime}(c_1)\vert_{H \cap Q}$ 
and this contradicts the semistability of 
${\mathcal E}^{\prime}\vert_{H \cap Q}$.   
\end{proof}

\begin{corol}\label{C:nounstablehyp}
Under the hypothesis of Prop.~\ref{P:nounstablehyp}, minus the assumption 
about $H_0$$:$   

\emph{(i)} For every line 
$L \subset Q$ one has $E\vert_L \simeq {\mathcal O}_L(a) \oplus 
{\mathcal O}_L(-a+c_1)$, with $a \leq c_2 + c_1$. 

\emph{(ii)} Assuming $c_2 > -c_1 + 1$, if $L_1,L_2 \subset Q$ 
are two distinct lines such that $E\vert_{L_i} \simeq 
{\mathcal O}_{L_i}(c_2+c_1) \oplus {\mathcal O}_{L_i}(-c_2)$, $i=1,2$, then 
$L_1\cap L_2 = \emptyset$.
\end{corol}

\begin{proof}
By Prop.~\ref{P:nounstablehyp}, the complement of ${\mathcal U}_{\text{ss}}(E)$ 
in ${\mathcal U}(E) = {\mathbb P}^{4\vee} \setminus Q^{\vee}$ consists of one 
point or is empty. Let $L,L^{\prime} \subset Q$ be two lines intersecting in a  
point $x$. The set of hyperplanes $H \subset {\mathbb P}^4$ containing 
$L \cup L^{\prime}$ is a line in ${\mathbb P}^{4\vee}$ which intersects 
$Q^{\vee}$ in only one point (namely, the one corresponding to the tangent 
hyperplane to $Q$ at $x$). It follows that, for a general $H$ as above, 
$H \cap Q$ is nonsingular and $E\vert_{H \cap Q}$ is semistable. 
One can apply, now, Lemma~\ref{L:jump} and Cor.~\ref{C:jump1}.    
\end{proof}

\begin{lem}\label{L:efrestrq2}
Let $Q_2 \subset {\mathbb P}^3$ be a smooth quadric, let $x$ be a point of 
${\mathbb P}^3$ and let $\mathcal F$ be a rank $2$ reflexive sheaf on 
${\mathbb P}^3$ defined by an exact sequence: 
\begin{equation*}
0 \longrightarrow {\mathcal O}_{{\mathbb P}^3}(-1) \longrightarrow 
{\mathcal O}_{{\mathbb P}^3}^{\oplus 3} \longrightarrow {\mathcal F} 
\longrightarrow 0
\end{equation*}
where the left morphism is the dual of an epimorphism 
${\mathcal O}_{{\mathbb P}^3}^{\oplus 3} \rightarrow {\mathcal I}_{\{x\}}(1)$.  

\emph{(i)} If $x \in Q_2$ then one has an exact sequence: 
\begin{equation*}
0 \longrightarrow {\mathcal F}\vert_{Q_2} \longrightarrow 
{\mathcal O}_{Q_2}(1,0) \oplus {\mathcal O}_{Q_2}(0,1) \longrightarrow 
{\mathcal O}_{\{x\}} \longrightarrow 0\, .
\end{equation*}

\emph{(ii)} If $x \notin Q_2$ then the zero scheme of any non-zero global 
section of ${\mathcal F}\vert_{Q_2}$ is $0$-dimensional, of degree $2$.   
Moreover, $({\mathcal F}\vert_{Q_2})(1,-1)$ and 
$({\mathcal F}\vert_{Q_2})(-1,1)$ have only one non-zero global section 
(up to multiplication by scalars) and these sections vanish nowhere.
\end{lem}

\begin{proof}
(i) Let $p_1,\, p_2 : Q_2 \rightarrow {\mathbb P}^1$ be the projections 
and put, as in Section~\ref{S:auxq2}, ${\overline {\mathcal S}} = 
{\mathcal O}_{Q_2}(-1,0) \oplus {\mathcal O}_{Q_2}(0,-1)$. Taking the direct 
sum of the exact sequences obtained by applying $p_1^{\ast}$ and $p_2^{\ast}$ 
to the exact sequence:
\begin{equation*}
0 \longrightarrow {\mathcal O}_{{\mathbb P}^1}(-1) \longrightarrow 
{\mathcal O}_{{\mathbb P}^1}^{\oplus 2} \longrightarrow 
{\mathcal O}_{{\mathbb P}^1}(1) \longrightarrow 0
\end{equation*}
one gets an exact sequence 
\begin{equation*}
0 \longrightarrow {\overline {\mathcal S}} \longrightarrow 
{\mathcal O}_{Q_2}^{\oplus 4} \longrightarrow 
{\overline {\mathcal S}}(1,1) \longrightarrow 0\, .
\end{equation*}   
On the other hand, $\{x\}$ is the intersection of the lines $p_1^{-1}(p_1(x))$ 
and $p_2^{-1}(p_2(x))$, hence one has an exact sequence: 
\begin{equation*}
0 \longrightarrow {\mathcal O}_{Q_2}(-1,-1) \longrightarrow 
{\overline {\mathcal S}} \longrightarrow {\mathcal I}_{\{x\},Q_2} 
\longrightarrow 0\, .
\end{equation*}
Applying the Snake Lemma to the diagram:
\begin{equation*}
\begin{CD}
0 @>>> {\mathcal O}_{Q_2} @>>> {\mathcal O}_{Q_2}^{\oplus 4} @>>> 
{\mathcal O}_{Q_2}^{\oplus 3} @>>> 0\\
@. @VVV @VVV @VVV \\
0 @>>> {\mathcal O}_{Q_2} @>>> {\overline {\mathcal S}}(1,1)   
@>>> {\mathcal I}_{\{x\},Q_2}(1,1) @>>> 0
\end{CD}
\end{equation*}
(in which the right vertical morphism is the evaluation 
morphism) one deduces an exact sequence:
\begin{equation*}
0 \longrightarrow  {\mathcal O}_{Q_2}(-1,0) \oplus {\mathcal O}_{Q_2}(0,-1) 
\longrightarrow {\mathcal O}_{Q_2}^{\oplus 3} \longrightarrow 
{\mathcal I}_{\{x\},Q_2}(1,1) \longrightarrow 0\, .
\end{equation*}
Applying ${\mathcal Hom}_{{\mathcal O}_{Q_2}}(-,{\mathcal O}_{Q_2})$ to this 
exact sequence one gets and exact sequence:
\begin{equation*}
0 \longrightarrow {\mathcal O}_{Q_2}(-1,-1) \longrightarrow 
{\mathcal O}_{Q_2}^{\oplus 3} \longrightarrow 
{\mathcal O}_{Q_2}(1,0) \oplus {\mathcal O}_{Q_2}(0,1) \longrightarrow 
{\mathcal O}_{\{x\}} \longrightarrow 0
\end{equation*}
because ${\mathcal Ext}^1_{{\mathcal O}_{Q_2}}({\mathcal I}_{\{x\},Q_2},
{\mathcal O}_{Q_2}) \simeq {\mathcal O}_{\{x\}}$. 

(ii) The zero scheme of any non-zero global section of $\mathcal F$  
is a line $L$ in ${\mathbb P}^3$ passing through $x$ and one has an 
exact sequence: 
\begin{equation*}
0 \longrightarrow {\mathcal O}_{{\mathbb P}^3} \longrightarrow 
{\mathcal F} \longrightarrow {\mathcal I}_L(1) \longrightarrow 0\, .
\end{equation*}
Such a line intersects $Q_2$ in a 0-dimensional scheme of degree 2.  
The first assertion from the 
second point of the conclusion is now clear. As for the last assertion, 
tensorizing by ${\mathcal O}_{Q_2}(1,-1)$ the exact sequence:
\begin{equation*}
0 \longrightarrow {\mathcal O}_{Q_2}(-1,-1) \longrightarrow 
{\mathcal O}_{Q_2}^{\oplus 3} \longrightarrow {\mathcal F}\vert_{Q_2}  
\longrightarrow 0
\end{equation*}
and taking global sections, one deduces that 
$\text{h}^0(({\mathcal F}\vert_{Q_2})(1,-1)) = 1$. 
Let $s$ be a non-zero global section of 
$({\mathcal F}\vert_{Q_2})(1,-1)$ and suppose that $s$ vanishes in a 
point $y \in Q_2$. Let $L$ be the line belonging to the linear system 
$\vert \, {\mathcal O}_{Q_2}(1,0)\, \vert$ passing through $y$. Since 
$x \notin L$, one sees easily that ${\mathcal F}\vert_L \simeq 
{\mathcal O}_L \oplus {\mathcal O}_L(1)$, hence 
$({\mathcal F}\vert_{Q_2})(1,-1)\vert_L \simeq {\mathcal O}_L(-1) 
\oplus {\mathcal O}_L$. Since $s$ vanishes in $y \in L$, $s$ must vanish on 
$L$. But: 
\begin{equation*}
({\mathcal F}\vert_{Q_2})(1,-1)\otimes {\mathcal O}_{Q_2}(-L) \simeq 
({\mathcal F}\vert_{Q_2})(0,-1)
\end{equation*}
and $\text{H}^0(({\mathcal F}\vert_{Q_2})(0,-1)) = 0$, hence $s = 0$, 
a contradiction. 
\end{proof} 

The next result, in which we use the method described in 
Section~\ref{S:standard}, is the key point in the proof of 
Theorem~\ref{T:nonstblhyp}. 

\begin{prop}\label{P:applstandconstr}
Let $\mathcal E$ be a stable rank $2$ reflexive sheaf on a nonsingular 
quadric threefold $Q = Q_3 \subset {\mathbb P}^4$, with 
${\fam0 det}\, {\mathcal E} \simeq {\mathcal O}_Q(c_1)$, $c_1 = 0$ or $-1$, 
and with $c_2 > 2 + c_1$. Let $\Sigma \subset {\mathbb P}^{4\vee}$ be a 
reduced and irreducible hypersurface, $\Sigma \neq Q^{\vee}$ and 
$\Sigma \neq x^{\vee}$, $\forall x \in {\fam0 Sing}\, {\mathcal E}$. 

If ${\mathcal U}_{\fam0 s}(\mathcal E) \cap 
{\mathcal U}_{\fam0 gm}(\mathcal E) \cap \Sigma = \emptyset$ then, 
$\forall h \in {\mathcal U}_{\fam0 ss}(\mathcal E) \cap \Sigma$, there exists 
a line $L \subset H \cap Q$ such that ${\mathcal E}\vert_L \simeq 
{\mathcal O}_L(c_2 + c_1) \oplus {\mathcal O}_L(-c_2)$. 
\end{prop}

\begin{proof}
Consider the incidence diagram from Lemma~\ref{L:relomega} for $n = 4$ and 
$d = 3$, i.e, for ${\mathbb P}^n = {\mathbb P}^4$ and 
${\mathbb G}_d({\mathbb P}^n) = {\mathbb P}^{4\vee}$. Let 
${\Sigma}_{\text{reg}}$ be the set of nonsingular points of $\Sigma$, let 
${\Sigma}^{\prime} := {\mathcal U}_{\text{ss}}(\mathcal E) \cap 
{\mathcal U}_{\text{gm}}(\mathcal E) \cap {\Sigma}_{\text{reg}}$, let 
$X^{\prime} := p^{-1}(Q) \cap q^{-1}({\Sigma}^{\prime})$ and consider the 
induced diagram: 
\begin{equation*}
\begin{CD}
X^{\prime} @>{q^{\prime}}>> {\Sigma}^{\prime}\\
@V{p^{\prime}}VV \\
Q
\end{CD}
\end{equation*}
Proposition \ref{P:gmss}(i) implies that $\Sigma^\prime \neq \emptyset$. 
By the proof of Lemma~\ref{L:gms}, it follows that, $\forall h \in 
{\Sigma}^{\prime}$, one has, in the case $c_1 = 0$, an exact sequence: 
\begin{equation*}
0 \longrightarrow {\mathcal O}_{H \cap Q} \longrightarrow 
{\mathcal E}\vert_{H \cap Q} \longrightarrow 
{\mathcal I}_{Z,H \cap Q} \longrightarrow 0 
\end{equation*}
where $Z$ is a 0-dimensional subscheme of $H \cap Q$ with $\text{deg}\, Z = 
c_2$, and, in the case $c_1 = -1$, one has an exact sequence of one of the 
forms:
\begin{gather*}
0 \longrightarrow {\mathcal O}_{H \cap Q}(-1,0) \longrightarrow 
{\mathcal E}\vert_{H \cap Q} \longrightarrow 
{\mathcal I}_{Z,H \cap Q}(0,-1) \longrightarrow 0 \\
0 \longrightarrow {\mathcal O}_{H \cap Q}(0,-1) \longrightarrow 
{\mathcal E}\vert_{H \cap Q} \longrightarrow 
{\mathcal I}_{Z,H \cap Q}(-1,0) \longrightarrow 0
\end{gather*}
where $Z$ is a 0-dimensional subscheme of $H \cap Q$ with $\text{deg}\, Z = 
c_2 - 1$. 

It follows that, in the case $c_1 = 0$ (resp., $c_1 = -1$),  
$q^{\prime}_{\ast}p^{\prime \ast}{\mathcal E}$ (resp., 
$q^{\prime}_{\ast}p^{\prime \ast}{\mathcal Hom}_{{\mathcal O}_Q}
({\mathcal S},{\mathcal E})$) is a line bundle on ${\Sigma}^{\prime}$. In the 
case $c_1 = 0$, the image of the canonical morphism 
$q^{\prime \ast}q^{\prime}_{\ast}p^{\prime \ast}{\mathcal E} \rightarrow 
p^{\prime \ast}{\mathcal E}$ is a saturated subsheaf ${\mathcal E}^{\prime}$ of 
$p^{\prime \ast}{\mathcal E}$ such that 
${\mathcal E}^{\prime}\vert_{q^{\prime\, -1}(h)} \simeq 
{\mathcal O}_{H \cap Q}$, $\forall h \in {\Sigma}^{\prime}$. In the case 
$c_1 = -1$, the image of the composite morphism:
\begin{equation*}
p^{\prime \ast}{\mathcal S} \otimes 
q^{\prime \ast}q^{\prime}_{\ast}p^{\prime \ast}{\mathcal Hom}
({\mathcal S},{\mathcal E}) 
\longrightarrow 
p^{\prime \ast}{\mathcal S} \otimes 
p^{\prime \ast}{\mathcal Hom}({\mathcal S},{\mathcal E})
\longrightarrow 
p^{\prime \ast}{\mathcal E}
\end{equation*}
is a saturated subsheaf ${\mathcal E}^{\prime}$ of 
$p^{\prime \ast}{\mathcal E}$ such that 
${\mathcal E}^{\prime}\vert_{q^{\prime\, -1}(h)} \simeq 
{\mathcal O}_{H \cap Q}(-1,0)$ or to ${\mathcal O}_{H \cap Q}(0,-1)$, 
$\forall h \in {\Sigma}^{\prime}$. Let ${\mathcal E}^{\prime \prime} := 
p^{\prime \ast}{\mathcal E}/{\mathcal E}^{\prime}$. 

We have to consider two cases: 1) $\Sigma \neq x^{\vee}$, $\forall x \in Q$, 
and 2) $\Sigma = x_0^{\vee}$, for some $x_0 \in Q$. In Case 1, one applies 
Prop.~\ref{P:standconstr} to $p^{\prime} : X^{\prime} \rightarrow Y := 
Q \setminus \text{Sing}\, {\mathcal E}$ and to the exact sequence 
$0 \rightarrow {\mathcal E}^{\prime} \longrightarrow p^{\prime \ast} 
{\mathcal E} \rightarrow {\mathcal E}^{\prime \prime} \rightarrow 0$. 
In Case 2, one applies Prop.~\ref{P:standconstr} to the restriction of  
$p^{\prime} : X^{\prime}\setminus p^{\prime\, -1}(x_0) \rightarrow 
Y\setminus \{x_0\}$ and to the above exact sequence restricted to 
$X^{\prime}\setminus p^{\prime\, -1}(x_0)$. The hypothesis of 
Prop.~\ref{P:standconstr} is satisfied by Lemma~\ref{L:sigma}. Since there 
exists 
no saturated subsheaf ${\overline {\mathcal E}}^{\, \prime}$ of 
${\mathcal E}\vert_Y$ 
(resp., of ${\mathcal E}\vert_{Y\setminus \{x_0\}}$ in 
Case 2) such that ${\overline {\mathcal E}}^{\, \prime}\vert_{H \cap Q}$ 
(resp., ${\overline {\mathcal E}}^{\, \prime}\vert_{H \cap Q \setminus \{x_0\}}$  
in Case 2) is isomorphic to ${\mathcal O}_{H \cap Q}$ for $c_1 = 0$ 
and to ${\mathcal O}_{H \cap Q}(-1,0)$ or to ${\mathcal O}_{H \cap Q}(0,-1)$ 
for $c_1 = -1$ (resp., to their restrictions to $H \cap Q \setminus \{x_0\}$ 
in Case 2), $\forall h \in {\Sigma}^{\prime}$,  
one deduces, from Prop.~\ref{P:standconstr}, that the second 
fundamental form ${\mathcal E}^{\prime} \rightarrow {\Omega}_{X^{\prime}/Q} 
\otimes {\mathcal E}^{\prime \prime}$ has generically rank $\geq 1$. 

It follows that, for a general $h \in {\Sigma}^{\prime}$, the restriction of 
the second fundamental form to $q^{\prime \, -1}(h)$ has generically rank 
$\geq 1$. By Lemma~\ref{L:sigma}(i), this restriction can be identified with 
a morphism of one of the forms: 
\begin{gather*}
{\mathcal O}_{H\cap Q} \longrightarrow ({\mathcal F}\vert_{H\cap Q}) 
\otimes {\mathcal I}_{Z,H\cap Q} \  \  (\text{for}\  c_1 = 0)\\
{\mathcal O}_{H\cap Q}(-1,0) \longrightarrow ({\mathcal F}\vert_{H\cap Q}) 
\otimes {\mathcal I}_{Z,H\cap Q}(0,-1) \  \  (\text{for}\  c_1 = -1)\\
{\mathcal O}_{H\cap Q}(0,-1) \longrightarrow ({\mathcal F}\vert_{H\cap Q}) 
\otimes {\mathcal I}_{Z,H\cap Q}(-1,0) \  \  (\text{for}\  c_1 = -1)
\end{gather*}
where $\mathcal F$ is the rank 2 reflexive sheaf on $H \simeq {\mathbb P}^3$ 
defined by an exact sequence: 
\begin{equation*}
0 \longrightarrow {\mathcal O}_H(-1) \longrightarrow 
{\mathcal O}_H^{\oplus \, 3} \longrightarrow {\mathcal F} \longrightarrow 0
\end{equation*}
with the left morphism equal to the dual of an epimorphism 
${\mathcal O}_H^{\oplus \, 3} \rightarrow {\mathcal I}_{\{x\},H}(1)$, $x$ being 
the point of $H$ corresponding to the tangent hyperplane 
$\text{T}_h\Sigma \subset {\mathbb P}^{4\vee}$. From the fact that the 
restriction of the second fundamental form to $q^{\prime \, -1}(h)$  has 
generically rank $\geq 1$, one deduces that 
$({\mathcal F}\vert_{H \cap Q})^{\ast \ast}$ (for $c_1 = 0$) and 
$({\mathcal F}\vert_{H \cap Q})^{\ast \ast}(1,-1)$ or 
$({\mathcal F}\vert_{H \cap Q})^{\ast \ast}(-1,1)$ (for $c_1 = -1$) has a 
global section vanishing on $Z$. Using Lemma~\ref{L:efrestrq2} and the fact 
that $c_2 > 2 + c_1$, one excludes the case where $x \notin H \cap Q$. It 
remains that $x \in H \cap Q$ and that 
$({\mathcal F}\vert_{H \cap Q})^{\ast \ast} \simeq 
{\mathcal O}_{H \cap Q}(1,0) \oplus {\mathcal O}_{H \cap Q}(0,1)$. One deduces, 
in the case $c_1 = 0$, that 
$\text{H}^0({\mathcal I}_{Z,H \cap Q}(1,0)) \neq 0$ or 
$\text{H}^0({\mathcal I}_{Z,H \cap Q}(0,1)) \neq 0$, and, in the case 
$c_1 = -1$, one deduces that 
$\text{H}^0({\mathcal I}_{Z,H \cap Q}(1,0)) \neq 0$ resp., that 
$\text{H}^0({\mathcal I}_{Z,H \cap Q}(0,1)) \neq 0$. It results that $H \cap Q$ 
contains a line $L$ as in the statement. 

We have proved, so far, that, for a general $h \in {\mathcal U}_{\text{ss}}
(\mathcal E) \cap \Sigma$, the conclusion of the Proposition is fulfilled. 
But, by Cor.~\ref{C:jump2}, the set of the points 
$h \in {\mathcal U}_{\text{ss}}(\mathcal E) \cap \Sigma$ satisfying the 
conclusion of the Proposition is closed in 
${\mathcal U}_{\text{ss}}(\mathcal E) \cap \Sigma$, hence it must be the whole 
of ${\mathcal U}_{\text{ss}}(\mathcal E) \cap \Sigma$. 
\end{proof}

\begin{rmk}\label{R:correctioneinsols}
Let $\mathcal E$ be a stable rank two reflexive sheaf on $Q = Q_3 \subset 
{\mathbb P}^4$ with ${\fam0 det}\, {\mathcal E} \simeq {\mathcal O}_Q(c_1)$, 
$c_1 = 0$ or $-1$. If ${\mathcal U}_{\fam0 s}({\mathcal E}) \cap 
{\mathcal U}_{\fam0 gm}({\mathcal E}) = \emptyset$ then ${\mathcal E} \simeq 
{\mathcal S}$. 

{\it Indeed}, the case $c_1 = 0$ was settled in \cite[Thm.~1.6]{ES84}. Let us 
assume that $c_1 = -1$. Consider the incidence diagram from 
Lemma~\ref{L:relomega}, for $n = 4$ and $d = 3$. It induces a diagram: 
\begin{equation*}
\begin{CD}
X @>{\overline q}>> {\mathbb P}^{4\vee}\\ 
@V{\overline p}VV \\ 
Q
\end{CD}
\end{equation*}
where $X = p^{-1}(Q)$. It follows, from the first part of the proof of 
Prop.~\ref{P:applstandconstr}, that if $c_2({\mathcal E}) = c_2[L]$, with 
$c_2 > 1$, then there exists a line bundle ${\mathcal E}^{\prime}$ on 
${\overline q}^{\, -1}({\mathcal U}_{\fam0 ss}({\mathcal E}) \cap 
{\mathcal U}_{\fam0 gm}({\mathcal E}))$ such that 
${\mathcal E}^{\prime}\vert_{{\overline q}^{\, -1}(h)} \simeq 
{\mathcal O}_{H \cap Q}(-1,0)$ or to ${\mathcal O}_{H \cap Q}(0,-1)$, 
$\forall h \in {\mathcal U}_{\fam0 ss}({\mathcal E}) \cap  
{\mathcal U}_{\fam0 gm}({\mathcal E})$. But this is not possible because, 
$X$ being a ${\mathbb P}^3$-bundle over $Q$, $\text{Pic}\, X \simeq 
{\overline p}^{\, \ast}\, \text{Pic}\, Q \oplus 
{\overline q}^{\, \ast}\, \text{Pic}\, {\mathbb P}^{4\vee}$. 

It remains that $c_2 = 1$ and that ${\mathcal E}\vert_{H \cap Q} \simeq 
{\mathcal O}_{H \cap Q}(-1,0) \oplus {\mathcal O}_{H \cap Q}(0,-1)$, 
$\forall h \in {\mathcal U}_{\fam0 ss}({\mathcal E}) \cap  
{\mathcal U}_{\fam0 gm}({\mathcal E})$. One uses, now, the fact, which can 
be verified as in the proof of \cite[Theorem~2.11(ii)]{Ott88}, that if 
${\mathcal F}$ is a rank 2 reflexive sheaf on $Q$ such that 
${\mathcal F}\vert_{H \cap Q} \simeq {\mathcal S}\vert_{H \cap Q}$ for some 
hyperplane $H \subset {\mathbb P}^4$ avoiding the singular points of 
${\mathcal F}$ (but not necessarily cutting $Q$ transversally) then 
${\mathcal F} \simeq {\mathcal S}$.       
\end{rmk}

\begin{lem}\label{L:dimjump}
Let $\mathcal E$ be a stable rank two reflexive sheaf on $Q = Q_3 \subset 
{\mathbb P}^4$ with ${\fam0 det}\, {\mathcal E} \simeq {\mathcal O}_Q(c_1)$, 
$c_1 = 0$ or $-1$, and with $c_2 > - c_1$. Then the set of points $\ell \in 
{\mathbb P}(U) = {\mathbb P}^3$ corresponding to the lines $L \subset Q$ 
which either pass through a singular point of $\mathcal E$ or have the 
property that ${\mathcal E}\vert_L \simeq {\mathcal O}_L(a) \oplus 
{\mathcal O}_L(-a+c_1)$, $a \geq c_2 + c_1$, is a closed subset of 
${\mathbb P}^3$, of dimension $\leq 1$. 
\end{lem}

\begin{proof}
The set of points $\ell \in {\mathbb P}(U) = {\mathbb P}^3$ 
corresponding to the lines $L \subset Q$ passing through a singular point of 
$\mathcal E$ is a union of lines in ${\mathbb P}^3$ (see 
par.~\ref{SS:varlines}). Let $V \subseteq {\mathbb P}^3$ be the complement of 
this union of lines. By a semicontinuity argument, the set of points $\ell 
\in V$ such that ${\mathcal E}\vert_L \simeq {\mathcal O}_L(a) \oplus 
{\mathcal O}_L(-a+c_1)$, $a \geq c_2 + c_1$, is a closed subset of $V$. 
Consequently, the set $Z$ from the statement of the lemma is a closed subset 
of ${\mathbb P}^3$. 

Now, by \cite[Thm.~1.6]{ES84} (complete with 
Remark~\ref{R:correctioneinsols}), for a general hyperplane $H \subset 
{\mathbb P}^4$, $H \cap Q$ is smooth, contains no singular point of 
$\mathcal E$ and ${\mathcal E}\vert_{H \cap Q}$ is stable. We have seen in 
par.~\ref{SS:linesonhyp} that the set of points $\ell \in {\mathbb P}^3$ 
such that $L \subset H \cap Q$ is a union of two disjoint line ${\Lambda}_1 
\cup {\Lambda}_2$. By the last assertion of Lemma~\ref{L:jump}, 
$({\Lambda}_1 \cup {\Lambda}_2) \cap Z = \emptyset$, hence 
$\text{dim}\, Z \leq 1$. 
\end{proof}

\begin{prop}\label{P:pointinq}
Let $\mathcal E$ be a stable rank two reflexive sheaf on $Q = Q_3 \subset 
{\mathbb P}^4$ with ${\fam0 det}\, {\mathcal E} \simeq {\mathcal O}_Q(c_1)$, 
$c_1 = 0$ or $-1$, and with $c_2 > 2$. Consider a point $x_0 \in Q 
\setminus {\fam0 Sing}\, {\mathcal E}$. 

Then, for a general hyperplane $H \subset {\mathbb P}^4$ passing through 
$x_0$, ${\mathcal E}\vert_{H \cap Q}$ is stable.
\end{prop}

\begin{proof}
Let $\Sigma := x_0^{\vee} \subset {\mathbb P}^{4\vee}$. Assume that 
${\mathcal U}_{\fam0 s}(\mathcal E) \cap 
{\mathcal U}_{\fam0 gm}(\mathcal E) \cap \Sigma = \emptyset$. Then 
Prop.~\ref{P:applstandconstr} implies that, 
$\forall h \in {\mathcal U}_{\fam0 ss}(\mathcal E) \cap \Sigma$, $H \cap Q$ 
contains a line $L$ such that ${\mathcal E}\vert_L \simeq 
{\mathcal O}_L(c_2 + c_1) \oplus {\mathcal O}_L(-c_2)$. Let $Z \subset 
{\mathbb P}(U) = {\mathbb P}^3$ be the set of lines in $Q$ from 
Lemma~\ref{L:dimjump}. The lines in $Q$ passing through $x_0$ correspond to 
a line ${\Lambda}_{x_0} \subset {\mathbb P}^3$. 

We assert that ${\Lambda}_{x_0} \subseteq Z$. {\it Indeed}, otherwise 
${\Lambda}_{x_0} \cap Z$ is a finite set. It would follow, from 
Lemma~\ref{L:dimjump}, that the set $\mathcal H$ of hyperplanes $H \subset 
{\mathbb P}^4$ containing $x_0$ and a line $L \subset Q$ corresponding to a 
point $\ell \in Z$, has dimension $\leq 2$ (in ${\mathbb P}^{4\vee}$), which 
would contradict the fact that ${\mathcal U}_{\fam0 ss}(\mathcal E) \cap 
\Sigma \subseteq {\mathcal H}$. 

Now, let $h \in {\mathcal U}_{\fam0 ss}(\mathcal E) \cap \Sigma$. Let 
$L_1,L_2$ be the two lines in $H \cap Q$ passing through $x_0$. Since 
${\Lambda}_{x_0} \subseteq Z$, it follows that 
${\mathcal E}\vert_{L_i} \simeq {\mathcal O}_{L_i}(a_i) \oplus 
{\mathcal O}_{L_i}(-a_i+c_1)$, with $a_i \geq c_2 + c_1$, $i = 1,2$. Since 
$c_2 + c_1 > 1$, this contradicts Cor.~\ref{C:jump1}. 

It thus remains that 
${\mathcal U}_{\fam0 s}(\mathcal E) \cap 
{\mathcal U}_{\fam0 gm}(\mathcal E) \cap x_0^{\vee} \neq \emptyset$.   
\end{proof}

The next proposition concludes the proof of Thm.~\ref{T:nonstblhyp} (taking 
into account Prop.~\ref{P:pointinq} and the results quoted in the 
Introduction, before the statement of Thm.~\ref{T:nonstblhyp}). 

\begin{prop}\label{P:conclusion}
Let $E$ be a stable rank $2$ vector bundle on $Q = Q_3 \subset {\mathbb P}^4$ 
with ${\fam0 det}\, E \simeq {\mathcal O}_Q(c_1)$, $c_1 = 0$ or $-1$, and 
with $c_2 > 2$. Let $\Sigma \subset {\mathbb P}^{4\vee}$ be an irreducible 
hypersurface, $\Sigma \neq Q^{\vee}$ and $\Sigma \neq x^{\vee}$, 
$\forall x \in Q$. 

If ${\mathcal U}_{\fam0 s}(\mathcal E) \cap 
{\mathcal U}_{\fam0 gm}(\mathcal E) \cap \Sigma = \emptyset$ then $E$ is as 
in the statement of Thm.~\ref{T:nonstblhyp}$({\fam0 ii})$. 
\end{prop} 

\begin{proof}
Prop.~\ref{P:applstandconstr} implies that, 
$\forall h \in {\mathcal U}_{\fam0 ss}(\mathcal E) \cap \Sigma$, $H \cap Q$ 
contains a line $L$ such that $E\vert_L \simeq 
{\mathcal O}_L(c_2 + c_1) \oplus {\mathcal O}_L(-c_2)$. 
Cor.~\ref{C:nounstablehyp}(i) implies that the set $Z$ of the points 
$\ell \in {\mathbb P}(U) = {\mathbb P}^3$ corresponding to the lines 
$L \subset Q$ such that $E\vert_L \simeq 
{\mathcal O}_L(c_2 + c_1) \oplus {\mathcal O}_L(-c_2)$ is a closed subset of 
${\mathbb P}^3$. Since $\text{dim}({\mathcal U}_{\fam0 ss}(\mathcal E) \cap 
\Sigma) = 3$, one must have $\text{dim}\, Z = 1$ (taking into account 
Lemma~\ref{L:dimjump}). Choose two distinct points ${\ell}_1, {\ell}_2 \in Z$. 
By Cor.~\ref{C:nounstablehyp}(ii), $L_1 \cap L_2 = \emptyset$. Let $H_0 
\subset {\mathbb P}^4$ be the hyperplane spanned by $L_1$ and $L_2$. $H_0$ 
intersects $Q$ transversally (because $H_0 \cap Q$ contains two disjoint 
lines) and, by Cor.~\ref{C:jump1}, $E\vert_{H_0 \cap Q}$ is unstable. From 
Prop.~\ref{P:nounstablehyp} it follows that $\Sigma = (H_0 \cap Q)^{\vee}$. 
Let $\ell \in Z$ be another point and let $L \subset Q$ be the corresponding 
line. As above, $L \cap L_1 = \emptyset$ and if $H \subset {\mathbb P}^4$ is 
the hyperplane spanned by $L$ and $L_1$ then $\Sigma = (H \cap Q)^{\vee}$. One 
deduces that $H = H_0$. 

Consequently, the points of $Z$ correspond to the lines from one ruling of 
$H_0 \cap Q$. Assume that this ruling is the linear system 
$\vert \, {\mathcal O}_{H_0 \cap Q}(1,0)\, \vert$. Then there exists an integer 
$a$ such that $E\vert_{H_0 \cap Q}$ can be realized as an extension: 
\begin{equation*}
0 \longrightarrow {\mathcal O}_{H_0 \cap Q}(a,c_2+c_1) \longrightarrow 
E\vert_{H_0 \cap Q} \longrightarrow 
{\mathcal O}_{H_0 \cap Q}(-a+c_1,-c_2) \longrightarrow 0\, .
\end{equation*}
Computing Chern classes, one gets that:
\begin{equation*}
c_2 = c_2(E\vert_{H_0 \cap Q}) = -ac_2 + (c_2+c_1)(-a+c_1) = 
-(2c_2+c_1)a + (c_2+c_1)c_1\, ,
\end{equation*}
hence $(2c_2+c_1)a = (c_2+c_1)c_1 - c_2$. It follows that $c_1 = -1$ and 
$a = -1$. Let $E^{\prime}$ be the rank 2 vector bundle on $Q$ defined by the 
exact sequence:
\begin{equation}
\label{eprime}
0 \longrightarrow E^{\prime}(-1) \longrightarrow E \longrightarrow 
{\mathcal O}_{H_0 \cap Q}(0,-c_2) \longrightarrow 0\, .
\end{equation}
One has $\text{det}\, E^{\prime} \simeq {\mathcal O}_Q$. Let $c_2^{\prime} \in 
{\mathbb Z}$ be defined by $c_2(E^{\prime}) = c_2^{\prime}[L]$. Let $H \subset 
{\mathbb P}^4$ be a general hyperplane intersecting $Q$ transversally and 
such that $H \cap H_0 \cap Q$ is a smooth conic $C \simeq {\mathbb P}^1$. 
Restricting to $H \cap Q$ the exact sequence (\ref{eprime}) 
one gets an exact sequence: 
\begin{equation*}
0 \longrightarrow E^{\prime}(-1)\vert_{H \cap Q}  \longrightarrow 
E\vert_{H \cap Q} \longrightarrow 
{\mathcal O}_{{\mathbb P}^1}(-c_2) \longrightarrow 0\, .
\end{equation*} 
By Riemann-Roch on $H \cap Q$, $\chi (E\vert_{H \cap Q}) = -c_2 + 1$, 
and $\chi (E^{\prime}(-1)\vert_{H \cap Q}) = 
- c_2(E^{\prime}\vert_{H \cap Q}) = - c_2^{\prime}$. Since 
$\chi ({\mathcal O}_{{\mathbb P}^1}(-c_2)) = - c_2 + 1$, one deduces that 
$c_2^{\prime} = 0$. $E$ being stable, $E^{\prime}$ is semistable, hence 
$E^{\prime} \simeq {\mathcal O}_Q^{\oplus \, 2}$. Dualizing the exact sequence 
(\ref{eprime}) and twisting by $-1$, one gets an exact sequence:
\begin{equation*}
0 \longrightarrow E \longrightarrow {\mathcal O}_Q^{\oplus \, 2} 
\overset{\varepsilon}{\longrightarrow} {\mathcal O}_{H_0 \cap Q}(0,c_2) 
\longrightarrow 0\, .
\end{equation*}
The epimorphism $\varepsilon$ is defined by two global sections 
${\varepsilon}_1, {\varepsilon}_2$ of ${\mathcal O}_{H_0 \cap Q}(0,c_2)$ 
generating this sheaf. For a general choice of constants ${\alpha}_1, 
{\alpha}_2 \in {\mathbb C}$, the divisor $Y$ on $H_0 \cap Q$ associated to 
the global section ${\alpha}_1{\varepsilon}_1 + {\alpha}_2{\varepsilon}_2$ 
of ${\mathcal O}_{H_0 \cap Q}(0,c_2)$ consists of $c_2$ disjoint lines 
belonging to the linear system $\vert \, {\mathcal O}_{H_0 \cap Q}(0,1)\, 
\vert$. One sees easily that the image of the composite morphism 
\begin{equation*}
\begin{CD} 
E @>>> {\mathcal O}_Q^{\oplus \, 2} @>(-{\alpha}_2,{\alpha}_1)>>  
{\mathcal O}_Q 
\end{CD}
\end{equation*}
is exactly the ideal sheaf ${\mathcal I}_Y$ of $Y$ in $Q$. 
Consequently $E$ can be realized as an extension:
\begin{equation*}
0 \longrightarrow {\mathcal O}_Q(-1) \longrightarrow E \longrightarrow 
{\mathcal I}_Y \longrightarrow 0\, .
\qedhere 
\end{equation*}  
\end{proof} 

\begin{rmk*}
We take this opportunity to point out some simplifications in the proof of 
the main result of Coand\u{a}~\cite{Co92}. 

(i) The approach to the Standard Construction used in 
Section~\ref{S:standard} of the present paper clarifies, hopefully, the proof 
of \cite[Lemma~1]{Co92}. 

(ii) The proof of \cite[Lemma~2]{Co92} is too complicated. 
{\it Indeed}, 
as it was noticed by Vall\`{e}s~\cite{Va95}, one can use the following easy 
argument: let $L$ be a line in ${\mathbb P}^3$ such that $E\vert_L  
\simeq {\mathcal O}_L \oplus {\mathcal O}_L(c_1)$. Let $L^{\vee}$ be the line 
in ${\mathbb P}^{3\vee}$ consisting of the planes $H \subset {\mathbb P}^3$ 
containing $L$. Then, obviously, $L^{\vee} \cap W_1(E) = \emptyset$, hence 
$\text{dim}\, W_1(E) \leq 1$. 

(iii) More important, one can get rid of the Case B in the proof of 
\cite[Prop.~2]{Co92} and, consequently, there's no need of 
\cite[Lemma~6]{Co92}. {\it Indeed}, under the hypothesis of 
\cite[Prop.~2]{Co92}, the first part of the proof of \cite[Prop.~2]{Co92} 
(on page 105) and 
\cite[Lemma~5]{Co92} show that $E$ is a {\it mathematical instanton}. 
We {\it claim} that, in this case, the set $\Gamma$ of points of 
${\mathbb G}_1({\mathbb P}^3)$ corresponding to the lines $L \subset 
{\mathbb P}^3$ such that $E\vert_L \simeq {\mathcal O}_L(c_2) \oplus 
{\mathcal O}_L(-c_2)$ is a linear section of ${\mathbb G}_1({\mathbb P}^3)$ 
in its Pl\"{u}cker embedding in ${\mathbb P}^5$. 

Assume the claim, for the moment. Then, by \cite[Lemma~3]{Co92} and 
\cite[Remark~2]{Co92}, $\text{dim}\, \Gamma  = 1$ hence $\Gamma$ is a smooth 
conic, the union of two lines or a line. But, by \cite[Lemma~4(b)]{Co92}, the 
lines in ${\mathbb P}^3$ corresponding to the points of $\Gamma$ are mutually 
disjoint. Since the lines in ${\mathbb P}^3$ corresponding to a line in 
${\mathbb G}_1({\mathbb P}^3)$ are contained in a fixed plane and pass through 
a fixed point in that plane, it follows that $\Gamma$ is a smooth conic, 
i.e., the lines in ${\mathbb P}^3$ corresponding to the points of $\Gamma$ 
form one ruling of a nonsingular quadric surface. Consequently, only the 
Case A from the proof of \cite[Prop.~2]{Co92} can occur. 

Finally, let us {\it prove the above claim} (which is, actually, a well known 
fact). One has ${\mathbb P}^3 = {\mathbb P}(U)$, where $U = {\mathbb C}^4$. 
Let $L \subset {\mathbb P}^3$ be a line corresponding to a 2-dimensional 
vector subspace ${\mathbb C}u \oplus {\mathbb C}v$ of $U$. Using the coKoszul 
and Euler exact sequences recalled in par.~\ref{SS:varlines}, one sees 
easily that the image of the composite morphism: 
\begin{equation*}
\begin{CD}
U \otimes {\mathcal O}_{\mathbb P} @>>> 
\overset{2}{\textstyle \bigwedge}U \otimes {\mathcal O}_{\mathbb P}(1) 
@>{u\wedge v \wedge}>> 
\overset{4}{\textstyle \bigwedge}U \otimes {\mathcal O}_{\mathbb P}(1)
\end{CD}
\end{equation*}
is $\overset{4}{\bigwedge}U \otimes {\mathcal I}_L(1)$, hence one gets an 
exact sequence: 
\begin{equation*}
0 \longrightarrow ({\mathbb C}u \oplus {\mathbb C}v)\otimes 
{\mathcal O}_{\mathbb P} \longrightarrow \text{T}_{\mathbb P}(-1) 
\longrightarrow  
\overset{4}{\textstyle \bigwedge}U \otimes {\mathcal I}_L(1) 
\longrightarrow 0\, .
\end{equation*}
Tensorizing the Euler sequence by $E(-2)$ and 
taking into account that $E$ is a mathematical instanton, one gets that 
$\text{H}^1(\text{T}_{\mathbb P}(-3) \otimes E) \overset{\sim}{\rightarrow} 
\text{H}^2(E(-3))$ and $\text{H}^2(\text{T}_{\mathbb P}(-3) \otimes E) = 0$. 
Then, tensorizing the last exact sequence by $E(-2)$ one gets that: 
\begin{equation*}
\text{H}^1(\text{T}_{\mathbb P}(-3)\otimes E) \overset{\sim}{\longrightarrow} 
\overset{4}{\textstyle \bigwedge}U \otimes 
\text{H}^1({\mathcal I}_L \otimes E(-1))\  \  \text{and}\  \  
\text{H}^2({\mathcal I}_L \otimes E(-1)) = 0 \, .
\end{equation*}
Finally, tensorizing by $\overset{4}{\bigwedge}U\otimes_{\mathbb C}E(-1)$ the 
short exact sequence $0 \rightarrow {\mathcal I}_L \rightarrow 
{\mathcal O}_{\mathbb P} \rightarrow {\mathcal O}_L \rightarrow 0$, one deduces 
an exact sequence: 
\begin{equation*}
0 \rightarrow \overset{4}{\textstyle \bigwedge}U\otimes \text{H}^0(E_L(-1)) 
\rightarrow \text{H}^1(\text{T}_{\mathbb P}(-3)\otimes E) 
\overset{\psi (\ell)}{\longrightarrow} 
\overset{4}{\textstyle \bigwedge}U\otimes \text{H}^1(E(-1)) \rightarrow 
\overset{4}{\textstyle \bigwedge}U\otimes \text{H}^1(E_L(-1)) \rightarrow 0 
\end{equation*}
where $\psi (\ell)$ is the composite morphism: 
\begin{equation*}
\begin{CD}
\text{H}^1(\text{T}_{\mathbb P}(-3)\otimes E) @>>> 
\overset{2}{\textstyle \bigwedge}U\otimes \text{H}^1(E(-1)) 
@>{u \wedge v \wedge}>> 
\overset{4}{\textstyle \bigwedge}U\otimes \text{H}^1(E(-1))\, .
\end{CD}
\end{equation*}
Let $\psi$ be the composite morphism on ${\mathbb P}^5 = 
{\mathbb P}(\overset{2}{\bigwedge}U)$: 
\begin{gather*}
{\mathcal O}_{{\mathbb P}^5}(-1)\otimes 
\text{H}^1(\text{T}_{\mathbb P}(-3) \otimes E) 
\longrightarrow 
{\mathcal O}_{{\mathbb P}^5}\otimes \overset{2}{\textstyle \bigwedge}U \otimes 
\text{H}^1(T_{\mathbb P}(-3) \otimes E) \longrightarrow \\
\longrightarrow 
{\mathcal O}_{{\mathbb P}^5}\otimes \overset{2}{\textstyle \bigwedge}U \otimes 
\overset{2}{\textstyle \bigwedge}U \otimes \text{H}^1(E(-1)) \longrightarrow  
{\mathcal O}_{{\mathbb P}^5}\otimes \overset{4}{\textstyle \bigwedge}U \otimes 
\text{H}^1(E(-1))\, .
\end{gather*}
Since $\text{h}^1(E(-1)) = c_2 = \text{h}^2(E(-3))$, one deduces that the 
above defined set $\Gamma$ is the intersection of 
${\mathbb G}_1({\mathbb P}^3)$ with the zero scheme of the morphism $\psi$. 
\end{rmk*}

\section{Restrictions to singular hyperplane sections}\label{S:proof2}

In this section we prove the second theorem from the Introduction. We 
explicitate, firstly, the notion of stability for a rank 2 
vector bundle on a singular hyperplane section of a quadric threefold 
$Q\subset {\mathbb P}^4$ (which is a quadratic cone) in terms of the 
pull-back of the bundle on the desingularization of the cone. 
Then, we consider a simultaneous desingularization of the 
family of singular hyperplane sections of  
$Q$ and describe its sheaf of relative K\"{a}hler 
differentials over $Q$. These are preparatory results for the Standard 
Construction which we, finally, apply in order to get a 
proof of Theorem~\ref{T:singnonstblhyp}.   

\subsection{Stability on the quadratic cone}\label{SS:stabcone}
We use the notation from par.~\ref{SS:varlines}. Recall, also, that we 
follow {\it the classical convention for projective bundles}.  
Let $y$ be a point of $Q$ and let $Y := Q \cap \text{T}_yQ$. One, usually, 
desingularizes $Y$ by blowing-up its vertex $y$. In our situation, one can 
obtain, geometrically, this desingularization as it follows. We viewed $Q$ as 
the intersection of the Grassmannian ${\mathbb G}_1({\mathbb P}^3) 
\hookrightarrow {\mathbb P}^5$ by a hyperplane ${\mathbb P}^4\subset 
{\mathbb P}^5$. Recall the incidence diagram: 
\begin{equation}
\label{f01q}
\begin{CD}
{\mathbb F}_{0,1}(Q) @>q>> {\mathbb P}^3\\ 
@VpVV\\ 
Q
\end{CD}
\end{equation} 
If $\ell \in {\mathbb P}^3$ then $p$ maps $q^{-1}(\ell)$ isomorphically onto 
a line $L \subset Q$ (and, in this way, one gets all the lines contained in 
$Q$) and if $x\in Q$ then $q$ maps $p^{-1}(x)$ isomorphically onto a line 
$L_x\subset {\mathbb P}^3$ (and, in this way, one gets all the jumping lines 
of the null-correlation bundle $N_{\omega}$ on ${\mathbb P}^3$). Put 
${\widetilde Y} := q^{-1}(L_y)$. 
From diagram~(\ref{f01q}) one deduces an incidence diagram: 
\begin{equation*}
\begin{CD}
{\widetilde Y} @>{\pi}_y>> L_y \simeq {\mathbb P}^1\\
@V{\sigma}_yVV\\ 
Y
\end{CD}
\end{equation*}
Since ${\mathbb F}_{0,1}(Q) \simeq 
{\mathbb P}(N^{\ast}_{\omega}(-1))$ over ${\mathbb P}^3$ and since $L_y$ is 
a jumping line for the null-correlation bundle $N_{\omega}$, it follows that 
${\widetilde Y} \simeq {\mathbb P}({\mathcal O}_{L_y} \oplus 
{\mathcal O}_{L_y}(-2))$ over $L_y$ (such that 
${\mathcal O}_{{\mathbb P}({\mathcal O} \oplus {\mathcal O}(-2))}(-1) \simeq 
{\sigma}_y^{\ast}{\mathcal O}_Y(-1)$). $\pi$ maps isomorphically 
$C_y := {\sigma}_y^{-1}(y) = p^{-1}(y)$ onto $L_y$. As  
${\sigma}_y^{\ast}{\mathcal O}_Y(-1)\vert_{C_y} \simeq {\mathcal O}_{C_y}$, 
it follows that $C_y$ can be identified with ${\mathbb P}({\mathcal O}_{L_y}) 
\subset {\mathbb P}({\mathcal O}_{L_y}\oplus {\mathcal O}_{L_y}(-2))$. One 
deduces that: 
\begin{equation*}
{\mathcal O}_{\widetilde Y}(C_y) \simeq {\sigma}_y^{\ast}{\mathcal O}_Y(1) 
\otimes {\pi}_y^{\ast}{\mathcal O}_{L_y}(-2)
\end{equation*}    
({\it in general}, if $E$ is a vector bundle over a scheme $T$, if $f : 
{\mathbb P}(E) \rightarrow T$ is the associated projective bundle, and if 
$E^{\prime}$ is a vector subbundle of $E$ then ${\mathbb P}(E^{\prime}) 
\subset {\mathbb P}(E)$ is the zero scheme of the composite morphism 
${\mathcal O}_{{\mathbb P}(E)}(-1) \hookrightarrow f^{\ast}E \rightarrow 
f^{\ast}(E/E^{\prime})$). Moreover, the restriction of 
${\sigma}_y : {\widetilde Y} \setminus C_y \rightarrow Y \setminus \{y\}$ 
is an isomorphism. As $Y$ is a 
normal variety, one deduces easily (see, for example, the proof of 
\cite[III,~Cor.~11.4]{Ha77}) that 
${\mathcal O}_Y \overset{\sim}{\rightarrow} 
{\sigma}_{y\ast}{\mathcal O}_{\widetilde Y}$. Besides, since 
${\mathcal I}_{C_y}/{\mathcal I}_{C_y}^2 \simeq 
{\mathcal O}_{{\mathbb P}^1}(2)$, one derives, as in the proof of 
\cite[V,~Prop.~3.4]{Ha77}, that 
$\text{R}^1{\sigma}_{y\ast}{\pi}_y^{\ast}{\mathcal O}_{L_y}(a) = 0$ for 
$a\geq -1$.   

\begin{lem}\label{L:sigmapiOpm1}
Let $\ell$ be a point of $L_y$ and put ${\widetilde L} := {\pi}_y^{-1}(\ell) 
\subset {\widetilde Y}$. ${\sigma}_y$ maps $\widetilde L$ isomorphically 
onto a line $L \subset Y$. Then 
${\sigma}_{y\ast}{\pi}_y^{\ast}{\mathcal O}_{L_y}(-1)\simeq {\mathcal I}_{L,Y}$, 
${\sigma}_{y\ast}{\pi}_y^{\ast}{\mathcal O}_{L_y}(1)\simeq 
{\mathcal I}_{L,Y}(1)$, and ${\mathcal I}_{L,Y}$ is a rank 1 reflexive sheaf on 
$Y$.  
\end{lem}

\begin{proof}
Applying ${\sigma}_{y\ast}$ to the short exact sequence: 
\begin{equation*}
0 \longrightarrow {\pi}_y^{\ast}{\mathcal O}_{L_y}(-1) \longrightarrow 
{\mathcal O}_{\widetilde Y} \longrightarrow {\mathcal O}_{\widetilde L} 
\longrightarrow 0
\end{equation*}
one gets an exact sequence: 
\begin{equation*}
0 \longrightarrow {\sigma}_{y\ast}{\pi}_y^{\ast}{\mathcal O}_{L_y}(-1) 
\longrightarrow {\mathcal O}_Y \longrightarrow {\mathcal O}_L 
\longrightarrow 0
\end{equation*}
from which one deduces that 
${\sigma}_{y\ast}{\pi}_y^{\ast}{\mathcal O}_{L_y}(-1)\simeq {\mathcal I}_{L,Y}$. 
In order to prove the second isomorphism, one tensorizes by 
${\pi}_y^{\ast}{\mathcal O}_{L_y}(1)$ the exact sequence: 
\begin{equation*}
0 \longrightarrow {\mathcal O}_{\widetilde Y} \longrightarrow 
{\mathcal O}_{\widetilde Y}(C_y) \longrightarrow 
{\mathcal O}_{\widetilde Y}(C_y)\vert_{C_y} \longrightarrow 0
\end{equation*}
and one applies, then, ${\sigma}_{y\ast}$. Since 
$({\mathcal O}_{\widetilde Y}(C_y)\otimes {\pi}_y^{\ast}{\mathcal O}_{L_y}(1)) 
\vert_{C_y} \simeq {\mathcal O}_{{\mathbb P}^1}(-1)$, one deduces that: 
\begin{gather*}
{\sigma}_{y\ast}{\pi}_y^{\ast}{\mathcal O}_{L_y}(1)\simeq {\sigma}_{y\ast}
({\mathcal O}_{\widetilde Y}(C_y)\otimes {\pi}_y^{\ast}{\mathcal O}_{L_y}(1)) 
\simeq {\sigma}_{y\ast}({\sigma}_y^{\ast}{\mathcal O}_Y(1)\otimes 
{\pi}_y^{\ast}{\mathcal O}_{L_y}(-1)) \simeq\\
{\mathcal O}_Y(1)\otimes 
{\sigma}_{y\ast}{\pi}_y^{\ast}{\mathcal O}_{L_y}(-1)\simeq {\mathcal I}_{L,Y}(1) 
\, .
\end{gather*}
Finally, applying ${\sigma}_{y\ast}{\pi}_y^{\ast}$ to the exact sequence: 
\begin{equation*}
0 \longrightarrow {\mathcal O}_{L_y}(-1) \longrightarrow 
{\mathcal O}_{L_y}^{\oplus 2} \longrightarrow {\mathcal O}_{L_y}(1) 
\longrightarrow 0\, ,
\end{equation*}
one gets an exact sequence: 
\begin{equation*}
0 \longrightarrow {\mathcal I}_{L,Y} \longrightarrow {\mathcal O}_Y^{\oplus 2} 
\longrightarrow {\mathcal I}_{L,Y}(1) \longrightarrow 0\, .
\end{equation*}
Since ${\mathcal I}_{L,Y}(1)$ is torsion free one deduces that 
${\mathcal I}_{L,Y}$ is reflexive (see \cite[Prop.~1.1]{Ha80}). 
\end{proof}

\begin{lem}\label{L:rk1reflexive}
Let $L \subset Y$ be a line. If $\mathcal L$ is a rank 1 reflexive sheaf on 
$Y$ then there exists $a \in {\mathbb Z}$ such that ${\mathcal L} \simeq 
{\mathcal O}_Y(a)$ or ${\mathcal L} \simeq {\mathcal I}_{L,Y}(a)$. Moreover, 
the dual of ${\mathcal I}_{L,Y}$ is ${\mathcal I}_{L,Y}(1)$. 
\end{lem}

\begin{proof}
One deduces from \cite[Prop.~1.6]{Ha80} that the map ${\mathcal L} \mapsto 
{\mathcal L}\vert_{Y\setminus \{y\}}$ is a bijection between the set of 
isomorphism classes of rank 1 reflexive sheaves on $Y$ and 
$\text{Pic}(Y\setminus \{y\})$ (the inverse bijection being $M \mapsto 
j_{\ast}M$, where $j : Y\setminus \{y\} \rightarrow Y$ is the inclusion map). 
On the other hand, if $C\subset Y\setminus \{y\}$ is a smooth conic then, 
applying \cite[II,~Prop.~6.5]{Ha77} and \cite[III,~Ex.~12.5]{Ha77} to 
$\widetilde Y$ and $C_y$, one deduces that the restriction map 
$\text{Pic}(Y\setminus \{y\}) \rightarrow \text{Pic}(C)$ is an isomorphism.  
We notice, now, that ${\mathcal O}_Y(1)\vert_C \simeq 
{\mathcal O}_{{\mathbb P}^1}(2)$, that ${\mathcal I}_{L,Y}\vert_C \simeq 
{\mathcal O}_{{\mathbb P}^1}(-1)$ and that ${\mathcal Hom}_{{\mathcal O}_Y}
({\mathcal I}_{L,Y},{\mathcal O}_Y)\vert_C \simeq {\mathcal I}_{L,Y}(1) 
\vert_C$. Fixing an isomorphism ${\mathbb P}^1 \overset{\sim}{\rightarrow} C$, 
one gets that if ${\mathcal L}\vert_C \simeq {\mathcal O}_{{\mathbb P}^1}(2a)$ 
then ${\mathcal L} \simeq {\mathcal O}_Y(a)$ and if ${\mathcal L}\vert_C 
\simeq {\mathcal O}_{{\mathbb P}^1}(2a-1)$ then ${\mathcal L} \simeq 
{\mathcal I}_{L,Y}(a)$.   
\end{proof}

\begin{prop}\label{P:nonstblonthecone}
Let $F$ be a rank 2 vector bundle on $Y$, with ${\fam0 det}\, F \simeq 
{\mathcal O}_Y(c_1)$, $c_1 = 0$ or $-1$. Assume that, for a general line 
$L\subset Y$, $F\vert_L \simeq {\mathcal O}_L\oplus {\mathcal O}_L(c_1)$ 
and that, for a general conic $C\subset Y\setminus \{y\}$, 
$F\vert_C \simeq {\mathcal O}_{{\mathbb P}^1}(c_1)^{\oplus 2}$. 

\emph{(i)} If $c_1 = 0$ and $F$ is not stable then ${\sigma}_y^{\ast}F$ can be 
realized as an extension: 
\begin{equation*}
0\longrightarrow {\mathcal O}_{\widetilde Y} \longrightarrow 
{\sigma}_y^{\ast}F \longrightarrow {\mathcal I}_Z \longrightarrow 0
\end{equation*}
where $Z$ is a $0$-dimensional (or empty) subscheme of $\widetilde Y$ with 
$Z\cap C_y = \emptyset$. 

\emph{(ii)} If $c_1 = -1$ and $F$ is not stable then ${\sigma}_y^{\ast}F$ can 
be realized as an extension: 
\begin{equation*}
0 \longrightarrow {\pi}_y^{\ast}{\mathcal O}_{L_y}(-1) \longrightarrow 
{\sigma}_y^{\ast}F \longrightarrow {\mathcal I}_Z\otimes 
{\sigma}_y^{\ast}{\mathcal O}_Y(-1)\otimes {\pi}_y^{\ast}{\mathcal O}_{L_y}(1) 
\longrightarrow 0
\end{equation*}
where $Z$ is a $0$-dimensional (or empty) subscheme of $\widetilde Y$ such 
that $Z\cap C_y = \emptyset$ or $Z\cap C_y = 1$ simple point.
\end{prop} 

\begin{proof}
By stability we understand, of course, the Mumford-Takemoto stability with 
respect to ${\mathcal O}_Y(1)$. A general member of the linear system 
$\vert \, {\mathcal O}_Y(1)\, \vert$ is a smooth conic $C\subset Y\setminus 
\{y\}$. 

(i) Since $F$ is not stable it must contain a rank 1 reflexive sheaf 
$\mathcal L$ with $\text{deg}({\mathcal L}\vert_C) = 0$. By 
Lemma~\ref{L:rk1reflexive}, ${\mathcal L}\simeq {\mathcal O}_Y$ hence  
$\text{H}^0(F) \neq 0$. As ${\sigma}_{y\ast}{\sigma}_y^{\ast}F \simeq F$, it 
follows that $\text{H}^0({\sigma}_y^{\ast}F) \neq 0$. Let $s$ be a non-zero 
global section of ${\sigma}_y^{\ast}F$ and let $Z\subset {\widetilde Y}$ be 
its zero scheme. 

We show, firstly, that $Z\cap C_y = \emptyset$. Indeed, let $\ell$ be a point 
of $L_y$ and ${\widetilde L} := {\pi}_y^{-1}(\ell)\subset {\widetilde Y}$. 
If $\ell$ is general then $s\vert_{{\widetilde L}} \neq 0$. Since, by 
hypothesis, ${\sigma}_y^{\ast}F\vert_{{\widetilde L}} \simeq 
{\mathcal O}_{{\mathbb P}^1}^{\oplus 2}$, it follows that $s$ vanishes at no 
point of $\widetilde L$. It particular, it does not vanish at ${\widetilde L} 
\cap C_y$. Since ${\sigma}_y^{\ast}F\vert_{C_y} \simeq   
{\mathcal O}_{{\mathbb P}^1}^{\oplus 2}$, one deduces that $s$ vanishes at no 
point of $C_y$, hence $Z\cap C_y = \emptyset$. 

Next, we show that $\text{dim}\, Z \leq 0$. Indeed, assume that there exists 
an effective irreducible divisor $D \subset {\widetilde Y}$ such that $s$ 
vanishes on $D$. As divisors, $D \sim aC_y+b{\widetilde L}$. Since 
$D\cap {\widetilde L} = \emptyset$, $a = (D\cdot {\widetilde L}) = 0$. 
Since $D \cap C_y = \emptyset$, $b = b-2a = (D\cdot C_y) = 0$. It follows 
that $D \sim 0$, a contradiction. 

Consequently, ${\sigma}_y^{\ast}F$ can be realized as an extension: 
\begin{equation*}
0\longrightarrow {\mathcal O}_{\widetilde Y} \overset{s}{\longrightarrow}  
{\sigma}_y^{\ast}F \longrightarrow {\mathcal I}_Z \longrightarrow 0\, .
\end{equation*} 

(ii) It follows, as in (i), that $F$ has a subsheaf isomorphic to 
${\mathcal I}_{L,Y}$, $L$ being a (any) line contained in $Y$. From 
Lemma~\ref{L:sigmapiOpm1} and the last part of Lemma~\ref{L:rk1reflexive}: 
\begin{equation*}
{\sigma}_{y\ast}({\sigma}_y^{\ast}F\otimes {\pi}_y^{\ast}{\mathcal O}_{L_y}(1)) 
\simeq F\otimes {\mathcal I}_{L,Y}(1) \simeq 
{\mathcal Hom}_{{\mathcal O}_Y}({\mathcal I}_{L,Y},F)\, .
\end{equation*}
One deduces that $\text{H}^0({\sigma}_y^{\ast}F\otimes {\pi}_y^{\ast} 
{\mathcal O}_{L_y}(1)) \neq 0$. Let $s$ be a non-zero global section of 
${\sigma}_y^{\ast}F\otimes {\pi}_y^{\ast}{\mathcal O}_{L_y}(1)$ and 
$Z \subset {\widetilde Y}$ its zero scheme. 

We show, firstly, that $Z\cap C_y$ is empty or consists of one simple point. 
Indeed, for a general $\ell \in L_y$, the restriction of $s$ to 
${\widetilde L}:= {\pi}_y^{-1}(\ell)$ is non-zero. Since 
\begin{equation*}
{\sigma}_y^{\ast}F\otimes {\pi}_y^{\ast}{\mathcal O}_{L_y}(1)\vert_{{\widetilde L}} 
\simeq {\mathcal O}_{{\mathbb P}^1} \oplus 
{\mathcal O}_{{\mathbb P}^1}(-1)
\end{equation*} 
it follows that $s$ vanishes at no point of $\widetilde L$. In particular, it 
does not vanish at ${\widetilde L} \cap C_y$. As: 
\begin{equation*}
{\sigma}_y^{\ast}F\otimes {\pi}_y^{\ast}{\mathcal O}_{L_y}(1)\vert_{C_y}
\simeq {\mathcal O}_{{\mathbb P}^1}(1)^{\oplus 2} 
\end{equation*} 
one deduces that $Z\cap C_y$ is empty or consists of one simple point. 

Next, we show that $Z$ is 0-dimensional or empty. Indeed, assume that there 
exists an effective irreducible divisor $D \subset {\widetilde Y}$ such that 
$s$ vanishes on $D$. As a divisor, $D\sim aC_y+b{\widetilde L}$. 
$D\cap {\widetilde L} = \emptyset$ implies that $a = (D\cdot {\widetilde L}) 
= 0$. Since $D \cap C_y$ is empty or consists of one simple point, it follows 
that $b = b-2a = (D\cdot C_y) \in \{0,1\}$. One deduces that $D$ must be 
a fibre of ${\pi}_y : {\widetilde Y} \rightarrow L_y$. But this would imply 
that ${\sigma}_y^{\ast}F \simeq ({\sigma}_y^{\ast}F\otimes {\pi}_y^{\ast}
{\mathcal O}_{L_y}(1))\otimes {\mathcal O}_{\widetilde Y}(-D)$ has a non-zero 
global section, which contradicts the fact that $\text{H}^0(F) = 0$ 
(because $F\vert_C \simeq {\mathcal O}_{{\mathbb P}^1}(-1)^{\oplus 2}$). 

Consequently, ${\sigma}_y^{\ast}F$ can be realized as an extension:
\begin{equation*}
0 \longrightarrow {\pi}_y^{\ast}{\mathcal O}_{L_y}(-1) 
\overset{s}{\longrightarrow}  
{\sigma}_y^{\ast}F \longrightarrow {\mathcal I}_Z\otimes 
{\sigma}_y^{\ast}{\mathcal O}_Y(-1)\otimes {\pi}_y^{\ast}{\mathcal O}_{L_y}(1) 
\longrightarrow 0\, .  \qedhere 
\end{equation*}   
\end{proof}

\begin{corol}\label{C:nonstblonthecone}
Under the hypothesis of Prop.~\ref{P:nonstblonthecone}\emph{(ii)}, one has 
${\fam0 h}^0({\sigma}_y^{\ast}F\otimes {\pi}_y^{\ast}{\mathcal O}_{L_y}(1)) 
= 1$. 
\end{corol}

\begin{proof}
One tensorizes the exact sequence from the conclusion of 
Prop.~\ref{P:nonstblonthecone}(ii) by ${\pi}_y^{\ast}{\mathcal O}_{L_y}(1)$ and 
one uses the fact that ${\sigma}_y^{\ast}{\mathcal O}_Y(-1) \otimes 
{\pi}_y^{\ast}{\mathcal O}_{L_y}(2) \simeq {\mathcal O}_{\widetilde Y}(-C_y)$.  
\end{proof}

\begin{rmk*}
For general results concerning the structure of rank 2 vector bundles on 
ruled surfaces, one may consult the papers of Brossius~\cite{Bro83} and 
Br\^{\i}nz\u{a}nescu~\cite{Br91}. 
\end{rmk*}

\subsection{Desingularization of the family of singular hyperplane 
sections}\label{SS:desing}
Consider the subset $X\subset Q\times Q$ defined by $X := \{(x,y)\, \vert \, 
\overline{xy}\subset Q\}$, where $\overline{xy}$ is the linear span of 
$\{x,y\}$ in ${\mathbb P}^4$. If the equation of $Q \subset {\mathbb P}^4$ is 
${\mathbf q}(x,x) = 0$, where ${\mathbf q} : {\mathbb C}^5\times {\mathbb C}^5 
\rightarrow {\mathbb C}$ is a non-degenerate symmetric bilinear form, then: 
\begin{equation*}
X = (Q\times Q)\cap \{(x,y)\in {\mathbb P}^4\times {\mathbb P}^4\, \vert \, 
{\mathbf q}(x,y) = 0\}\, .
\end{equation*} 
Let $p_1,\, p_2 : X \rightarrow Q$ be the restrictions of the canonical 
projections. If $y\in Q$ then $p_1 : X \rightarrow Q$ maps $p_2^{-1}(y)$ 
isomorphically onto $Y := Q\cap \text{T}_yQ$. One parametrizes, in this way, 
the singular hyperplane sections of $Q$ in ${\mathbb P}^4$. 

Recalling the diagram~(\ref{f01q}), consider the fibre product 
${\widetilde X} := {\mathbb F}_{0,1}(Q)\times_{{\mathbb P}^3}
{\mathbb F}_{0,1}(Q)$. Since 
\begin{equation*}
{\mathbb F}_{0,1}(Q) = \{(x,\ell) \in Q\times {\mathbb P}^3\, \vert \, 
x\in L\}\, ,
\end{equation*}
it follows that 
\begin{equation*}
{\widetilde X} = \{(x,y,\ell) \in Q\times Q\times {\mathbb P}^3\, \vert \, 
x\in L,\, y\in L\}\, .
\end{equation*}

Let ${\widetilde p}_1,\, {\widetilde p}_2 : {\widetilde X}\rightarrow Q$ be 
the restrictions of the canonical projections $\text{pr}_1, \text{pr}_2 : 
Q\times Q\times {\mathbb P}^3 \rightarrow Q$, and let $\pi : {\widetilde X} 
\rightarrow {\mathbb P}^3$ be the restriction of the canonical projection 
$\text{pr}_3 : Q\times Q\times {\mathbb P}^3 \rightarrow {\mathbb P}^3$. The 
canonical projection $\text{pr}_{12} : Q\times Q\times {\mathbb P}^3 
\rightarrow Q\times Q$ induces a morphism $\sigma : {\widetilde X} \rightarrow 
X$. If $\Delta$ is the diagonal of $Q\times Q$ then ${\sigma}^{-1}(\Delta) = 
\{(x,x,\ell)\, \vert \, x\in L\} \simeq {\mathbb F}_{0,1}(Q)$ and 
$\sigma$ induces an isomorphism 
${\widetilde X}\setminus {\sigma}^{-1}(\Delta) 
\overset{\sim}{\rightarrow} X\setminus \Delta$. Finally, the canonical 
projections $\text{pr}_{13}, \text{pr}_{23} : Q\times Q\times {\mathbb P}^3 
\rightarrow Q\times {\mathbb P}^3$ induce morphisms $p_{13}, p_{23} : 
{\widetilde X} \rightarrow {\mathbb F}_{0,1}(Q)$ and one gets a diagram: 
\begin{equation}
\label{tildex}
\begin{CD}
{\widetilde X} @>{p_{23}}>> {\mathbb F}_{0,1} @>p>> Q\\ 
@V{p_{13}}VV @VVqV\\
{\mathbb F}_{0,1}(Q) @>q>> {\mathbb P}^3\\ 
@VpVV\\
Q
\end{CD}
\end{equation}
having a cartesian square and with ${\widetilde p}_1 = p\circ p_{13} = p_1\circ 
\sigma$, ${\widetilde p}_2 = p \circ p_{23} = p_2 \circ \sigma$, 
$\pi = q \circ p_{13} = q \circ p_{23}$. 

Now, let $y$ be a point of $Q$. $p_{13} : {\widetilde X} \rightarrow 
{\mathbb F}_{0,1}(Q)$ maps ${\widetilde p}_2^{\, -1}(y)$ isomorphically onto the 
desingularization $\widetilde Y$ of $Y := Q\cap \text{T}_yQ$ considered in 
par.~\ref{SS:stabcone}. Under this identification, the restriction of 
$\sigma : {\widetilde X} \rightarrow X$ (resp., $\pi : {\widetilde X} 
\rightarrow {\mathbb P}^3$) to ${\widetilde p}_2^{\, -1}(y)$ is identified with 
the morphism ${\sigma}_y : {\widetilde Y} \rightarrow Y$ (resp., 
${\pi}_y : {\widetilde Y} \rightarrow L_y$) from par.~\ref{SS:stabcone}. 

Let ${\Omega}_{{\widetilde p}_1}$ be the sheaf of relative K\"{a}hler 
differentials of the morphism ${\widetilde p}_1 : {\widetilde X} \rightarrow 
Q$. We want to describe, in this paragraph, the restriction of 
${\Omega}_{{\widetilde p}_1}$ to ${\widetilde p}_2^{\, -1}(y) \simeq 
{\widetilde Y}$. 

\begin{lem}\label{L:fy}
Let $H := {\fam0 T}_yQ \simeq {\mathbb P}^3$ and let ${\mathcal F}_y$ be the 
${\mathcal O}_H$-module defined by the exact sequence: 
\begin{equation*}
0 \longrightarrow {\mathcal O}_H(-1) \longrightarrow {\mathcal O}_H^{\oplus 3} 
\longrightarrow {\mathcal F}_y \longrightarrow 0
\end{equation*}
where the left morphism is the dual of an epimorphism 
${\mathcal O}_H^{\oplus 3} \rightarrow {\mathcal I}_{\{y\},H}(1)$. Then 
${\Omega}_{{\widetilde p}_1}\vert_{{\widetilde Y}\setminus C_y}$ can be 
identified, via the isomorphism ${\widetilde Y}\setminus C_y 
\overset{\sim}{\rightarrow} Y\setminus \{y\}$ induced by $\sigma_y$,  
with ${\mathcal F}_y\vert_{Y\setminus \{y\}}$. 
\end{lem} 

\begin{proof}
${\Omega}_{{\widetilde p}_1}\vert_{{\widetilde Y}\setminus C_y}$ can be  
identified with ${\Omega}_{p_1}\vert_{Y\setminus \{y\}}$ and, by 
Lemma~\ref{L:sigma}(i), ${\Omega}_{p_1}\vert_Y$ can be 
identified with ${\mathcal F}_y\vert_Y$. 
\end{proof}

\begin{rmk}\label{R:h1el}
Let $T$ be a scheme, $E$ a vector bundle on $T$ and $f : {\mathbb P}(E) 
\rightarrow T$ the associated (classical) projective bundle. If $L$ is a line 
bundle on $T$ then $\text{R}^if_{\ast}({\mathcal O}_{{\mathbb P}(E)}(1)\otimes 
f^{\ast}L) = 0$ for $i \geq 1$. One deduces that the canonical morphism: 
\begin{equation*}
\text{H}^1(E^{\ast}\otimes L) \simeq 
\text{H}^1(f_{\ast}({\mathcal O}_{{\mathbb P}(E)}(1)\otimes f^{\ast}L)) 
\longrightarrow \text{H}^1({\mathcal O}_{{\mathbb P}(E)}(1)\otimes f^{\ast}L)
\end{equation*}  
is an isomorphism.

Now, let $E^{\prime}$ be a vector subbundle of $E$ and let 
$f^{\prime} : {\mathbb P}(E^{\prime}) \rightarrow T$ be the associated 
projective bundle. One has 
\begin{equation*}
({\mathcal O}_{{\mathbb P}(E)}(1)\otimes f^{\ast}L)\vert_{{\mathbb P}(E^{\prime})} 
\simeq {\mathcal O}_{{\mathbb P}(E^{\prime})}(1)\otimes 
f^{\prime \ast}L
\end{equation*}
and, from the above observation, the restriction morphism: 
\begin{equation*}
\text{H}^1({\mathcal O}_{{\mathbb P}(E)}(1)\otimes f^{\ast}L) \longrightarrow 
\text{H}^1({\mathcal O}_{{\mathbb P}(E^{\prime})}(1)\otimes f^{\prime \ast}L)
\end{equation*} 
can be identified with the morphism $\text{H}^1(E^{\ast}\otimes L) \rightarrow 
\text{H}^1(E^{\prime \ast}\otimes L)$. 
\end{rmk}

\begin{prop}\label{P:omegap1}
One has a commutative diagram with exact rows: 
\begin{equation*}
\begin{CD}
0 @>>> {\mathcal O}_{\widetilde Y} @>>> 
{\pi}_y^{\ast}({\mathcal O}_{L_y}(1)^{\oplus 2}) @>>> 
{\pi}_y^{\ast}{\mathcal O}_{L_y}(2) @>>> 0\\
@. @VVV @VVV @VV{\fam0 id}V \\
0 @>>> {\mathcal O}_{\widetilde Y}(C_y) @>>> 
{\Omega}_{{\widetilde p}_1}\vert_{{\widetilde p}_2^{\, -1}(y)} @>>> 
{\pi}_y^{\ast}{\mathcal O}_{L_y}(2) @>>> 0
\end{CD}
\end{equation*}
where the left vertical morphism is the canonical one. 
\end{prop} 

\begin{proof}
Using the diagram~(\ref{tildex}), one gets an exact sequence: 
\begin{equation*}
0 \longrightarrow p_{13}^{\ast}{\Omega}_p \longrightarrow 
{\Omega}_{{\widetilde p}_1} \longrightarrow 
{\Omega}_{p_{13}} \longrightarrow 0 
\end{equation*}
and  an isomorphism ${\Omega}_{p_{13}} \simeq p_{23}^{\ast}{\Omega}_q$. Since 
${\mathbb F}_{0,1}(Q) \simeq {\mathbb P}(\mathcal S)$ over $Q$ and 
${\mathbb F}_{0,1}(Q) \simeq {\mathbb P}(N_{\omega}^{\ast}(-1))$ over 
${\mathbb P}^3$, it follows, from \cite[III,~Ex.~8.4(b)]{Ha77} (recall that 
Hartshorne uses Grothendieck's convention for projective bundles), that 
${\Omega}_p \simeq p^{\ast}{\mathcal O}_Q(1)\otimes 
q^{\ast}{\mathcal O}_{{\mathbb P}^3}(-2)$ and ${\Omega}_q \simeq 
q^{\ast}{\mathcal O}_{{\mathbb P}^3}(2) \otimes p^{\ast}{\mathcal O}_Q(-2)$. 
One deduces an exact sequence: 
\begin{equation*}
0 \longrightarrow {\widetilde p}_1^{\, \ast}{\mathcal O}_Q(1) \otimes 
{\pi}^{\ast}{\mathcal O}_{{\mathbb P}^3}(-2) \longrightarrow 
{\Omega}_{{\widetilde p}_1} \longrightarrow 
{\pi}^{\ast}{\mathcal O}_{{\mathbb P}^3}(2) \otimes 
{\widetilde p}_2^{\, \ast}{\mathcal O}_Q(-2) \longrightarrow 0
\end{equation*}
which restricted to ${\widetilde p}_2^{\, -1}(y)$ gives us an exact sequence: 
\begin{equation}
\label{restomega}
0 \longrightarrow {\sigma}_y^{\ast}{\mathcal O}_Y(1) \otimes 
{\pi}_y^{\ast}{\mathcal O}_{L_y}(-2) \longrightarrow 
{\Omega}_{{\widetilde p}_1}\vert_{{\widetilde p}_2^{\, -1}(y)} 
\longrightarrow {\pi}_y^{\ast}{\mathcal O}_{L_y}(2) \longrightarrow 0 \, .
\end{equation}
This extension of line bundles correspondes to an element 
$\varepsilon \in \text{H}^1({\sigma}_y^{\ast}{\mathcal O}_Y(1) \otimes 
{\pi}_y^{\ast}{\mathcal O}_{L_y}(-4))$. Recalling that ${\widetilde Y} \simeq 
{\mathbb P}({\mathcal O}_{L_y}\oplus {\mathcal O}_{L_y}(-2))$ over $L_y$, one 
gets, from Remark~\ref{R:h1el}, a canonical isomorphism: 
\begin{equation*}
\text{H}^1({\sigma}_y^{\ast}{\mathcal O}_Y(1) \otimes 
{\pi}_y^{\ast}{\mathcal O}_{L_y}(-4)) \simeq 
\text{H}^1({\mathcal O}_{L_y}(-4) \oplus {\mathcal O}_{L_y}(-2))\, .
\end{equation*}

Now, $C_y \subset {\widetilde Y}$ can be (and was) identified with the 
projective subbundle ${\mathbb P}({\mathcal O}_{L_y})$ of 
${\mathbb P}({\mathcal O}_{L_y}\oplus {\mathcal O}_{L_y}(-2))$. One deduces, 
from the second part of Remark~\ref{R:h1el}, that the image of $\varepsilon$ 
by the canonical morphism $\text{H}^1({\mathcal O}_{L_y}(-4)\oplus 
{\mathcal O}_{L_y}(-2)) \rightarrow \text{H}^1({\mathcal O}_{L_y}(-4))$ 
corresponds to the restriction of the extension~(\ref{restomega}) to $C_y$. 
But $C_y := {\sigma}_y^{-1}(y) = {\widetilde p}_1^{\, -1}(y) \cap 
{\widetilde p}_2^{\, -1}(y)$ hence: 
\begin{equation*}
{\Omega}_{{\widetilde p}_1}\vert_{C_y} \simeq 
({\Omega}_{{\widetilde p}_1}\vert_{{\widetilde p}_1^{\, -1}(y)})\vert_{C_y} 
\simeq {\Omega}_{{\widetilde p}_1^{\, -1}(y)}\vert_{C_y} \, .
\end{equation*} 
Since $\pi : {\widetilde X} \rightarrow {\mathbb P}^3$ maps 
${\widetilde p}_1^{\, -1}(y)$ onto $L_y$ and induces an isomorphism 
$C_y \overset{\sim}{\rightarrow} L_y$, it follows that the inclusion 
$C_y \hookrightarrow {\widetilde p}_1^{\, -1}(y)$ admits a left inverse 
${\widetilde p}_1^{\, -1}(y) \rightarrow C_y$ which implies that 
${\Omega}_{C_y} \simeq {\mathcal O}_{{\mathbb P}^1}(-2)$ is a direct summand of 
${\Omega}_{{\widetilde p}_1^{\, -1}(y)}\vert_{C_y}$. One deduces that the 
restriction to $C_y$ of the exact sequence~(\ref{restomega}) {\it splits}, 
hence the image of $\varepsilon$ into $\text{H}^1({\mathcal O}_{L_y}(-4))$ is 
0. The kernel of the canonical morphism $\text{H}^1({\mathcal O}_{L_y}(-4)
\oplus {\mathcal O}_{L_y}(-2)) \rightarrow \text{H}^1({\mathcal O}_{L_y}(-4))$ 
is a 1-dimensional $\mathbb C$-vector space. Consequently, it remains only to 
decide whether the exact sequence~(\ref{restomega}) splits or not. 

Let $C\subset Y\setminus \{y\}$ be a smooth conic and let ${\widetilde C} := 
{\sigma}_y^{-1}(C) \subset {\widetilde Y}$. By Lemma~\ref{L:fy}, 
${\Omega}_{{\widetilde p}_1}\vert_{{\widetilde C}} \simeq {\mathcal F}_y \vert_C$. 
Let $H := \text{T}_yQ$ and let $P \subset H$ be the 2-plane 
spanned by $C$. One has ${\mathcal F}_y\vert_P \simeq \text{T}_P(-1)$ 
and $\text{T}_P(-1)\vert_C \simeq 
{\mathcal O}_{{\mathbb P}^1}(1)^{\oplus 2}$ 
(because $\text{H}^0(\text{T}_P(-2)\vert_C) = 0$) 
hence the restriction of (\ref{restomega}) to 
$\widetilde C$ {\it does not split}. 

Recalling that ${\sigma}_y^{\ast}{\mathcal O}_Y(1)\otimes {\pi}_y^{\ast} 
{\mathcal O}_{L_y}(-2) \simeq {\mathcal O}_{\widetilde Y}(C_y)$, one concludes 
that the exact sequence~(\ref{restomega}) is obtained by pushing-forward 
the exact sequence: 
\begin{equation*}
0 \longrightarrow {\mathcal O}_{\widetilde Y} \longrightarrow 
{\pi}_y^{\ast}({\mathcal O}_{L_y}(1)^{\oplus 2}) \longrightarrow 
{\pi}_y^{\ast}{\mathcal O}_{L_y}(2) \longrightarrow 0 
\end{equation*} 
via a non-zero morphism ${\mathcal O}_{\widetilde Y} \rightarrow 
{\mathcal O}_{\widetilde Y}(C_y)$. 
\end{proof}

\begin{corol}\label{C:omegap1}
Let $F$ denote the rank 2 vector bundle 
${\Omega}_{{\widetilde p}_1}\vert_{{\widetilde p}_2^{\, -1}(y)}$ on 
${\widetilde p}_2^{\, -1}(y) \simeq {\widetilde Y}$. 

\emph{(a)} If $s$ is a non-zero global section of $F$ and $Z\subset 
{\widetilde Y}$ its zero scheme then $Z \cap ({\widetilde Y}\setminus C_y) = 
\emptyset$ or there exists a fibre ${\widetilde L} = {\pi}_y^{-1}(\ell)$ of 
${\pi}_y : {\widetilde Y}\rightarrow L_y$ such that 
$Z \cap ({\widetilde Y}\setminus C_y) = {\widetilde L}\setminus C_y$ as closed 
subschemes of ${\widetilde Y}\setminus C_y$. 

\emph{(b)} ${\fam0 h}^0(F\otimes {\mathcal O}_{\widetilde Y}(-C_y)) = 1$ and  
a non-zero global section $s$ of $F\otimes {\mathcal O}_{\widetilde Y}(-C_y)$ 
vanishes nowhere. 
\end{corol}

\begin{proof}
(a) Applying the Snake Lemma to the commutative diagram in the statement of 
Prop.~\ref{P:omegap1} one gets an exact sequence: 
\begin{equation*}
0 \longrightarrow {\pi}_y^{\ast}({\mathcal O}_{L_y}(1)^{\oplus 2}) 
\longrightarrow F \longrightarrow {\mathcal O}_{\widetilde Y}(C_y)\vert_{C_y}  
\longrightarrow 0\, .
\end{equation*}
Since ${\mathcal O}_{\widetilde Y}(C_y)\vert_{C_y} \simeq 
{\mathcal O}_{{\mathbb P}^1}(-2)$, $s$ comes from a global section $s^{\prime}$ 
of ${\pi}_y^{\ast}({\mathcal O}_{L_y}(1)^{\oplus 2})$. The assertion from the 
statement becomes, now, obvious. 

(b) Tensorizing the bottom line of the diagram in the statement of 
Prop.~\ref{P:omegap1} by ${\mathcal O}_{\widetilde Y}(-C_y) \simeq 
{\sigma}_y^{\ast}{\mathcal O}_{\widetilde Y}(-1) \otimes {\pi}_y^{\ast} 
{\mathcal O}_{L_y}(2)$ one gets an exact sequence: 
\begin{equation*}
0 \longrightarrow {\mathcal O}_{\widetilde Y} \longrightarrow 
F\otimes {\mathcal O}_{\widetilde Y}(-C_y) \longrightarrow 
{\sigma}_y^{\ast}{\mathcal O}_{\widetilde Y}(-1) \otimes {\pi}_y^{\ast} 
{\mathcal O}_{L_y}(4) \longrightarrow 0\, .
\end{equation*} 
Since ${\pi}_{y\ast}({\sigma}_y^{\ast}{\mathcal O}_{\widetilde Y}(-1) \otimes 
{\pi}_y^{\ast}{\mathcal O}_{L_y}(4)) = 0$ it follows that 
$\text{H}^0({\sigma}_y^{\ast}{\mathcal O}_{\widetilde Y}(-1) \otimes 
{\pi}_y^{\ast}{\mathcal O}_{L_y}(4)) = 0$ and the assertion from the statement 
becomes, now, obvious. 
\end{proof}

\subsection{Use of the Standard Construction}\label{SS:proof2} 
We are, now, ready to give the 
\begin{proof}[Proof of Theorem~\ref{T:singnonstblhyp}] 
By \cite[Prop.~1.3]{ES84}, for a general line $L \subset Q$, one has 
${\mathcal E}\vert_L \simeq {\mathcal O}_L \oplus {\mathcal O}_L(c_1)$. 
By \cite[Cor. 1.5]{ES84} (see, also, Prop.~\ref{P:restrconics}), for a general 
conic $C \subset Q$, one has ${\mathcal E}\vert_C \simeq 
{\mathcal O}_{{\mathbb P}^1}(c_1)^{\oplus 2}$. Moreover, every line $L\subset Q$ 
and every conic $C \subset Q$ is contained in a singular hyperplane section 
of Q. 
Assume now that, for a general point $y \in Q$ as in the statement, the 
restriction of $\mathcal E$ to $Y := Q\cap \text{T}_yQ$ is not stable. Then,  
by Prop.~\ref{P:nonstblonthecone}, there exists a non-empty open subset 
$\mathcal V$ of $Q\setminus \bigcup_{x\in \text{Sing}\, {\mathcal E}}
\text{T}_xQ$ such that, for $y \in {\mathcal V}$, the rank 2 vector bundle 
${\sigma}_y^{\ast}{\mathcal E}$ on $\widetilde Y$ can be realized, in the case 
$c_1(\mathcal E) = 0$, as an extension: 
\begin{equation*}
0 \longrightarrow {\mathcal O}_{\widetilde Y} \longrightarrow 
{\sigma}_y^{\ast}{\mathcal E} \longrightarrow {\mathcal I}_Z \longrightarrow 0 
\end{equation*}
where $Z$ is a 0-dimensional subscheme of $\widetilde Y$, of length 
$c_2(\mathcal E)$, and such that $Z\cap C_y = \emptyset$, and, in the case 
$c_1(\mathcal E) = -1$, it can be realized as an extension:
\begin{equation*}
0 \longrightarrow {\pi}_y^{\ast}{\mathcal O}_{L_y}(-1) \longrightarrow 
{\sigma}_y^{\ast}{\mathcal E} \longrightarrow {\mathcal I}_Z\otimes 
{\sigma}_y^{\ast}{\mathcal O}_Y(-1) \otimes {\pi}_y^{\ast}{\mathcal O}_{L_y}(1) 
\longrightarrow 0
\end{equation*}
where $Z$ is a 0-dimensional subscheme of $\widetilde Y$, of length 
$c_2(\mathcal E)-1$, and such that $Z\cap C_y = \emptyset$ or 
$Z\cap C_y$ consists of a simple point. 
Consider the diagram: 
\begin{equation*}
\begin{CD}
{\overline X} @>{{\overline p}_2}>> {\mathcal V}\\
@V{{\overline p}_1}VV\\
Q\setminus \text{Sing}\, {\mathcal E}
\end{CD}
\end{equation*}
where ${\overline X} := {\widetilde p}_2^{\, -1}(\mathcal V) \subset 
{\widetilde X}$ and ${\overline p}_i$ is the restriction of 
${\widetilde p}_i$, $i = 1,2$. In the case $c_1(\mathcal E) = 0$, 
${\overline p}_1^{\ast}{\mathcal E}$ can be realized as an extension: 
\begin{equation*}
0 \longrightarrow {\mathcal L} \longrightarrow 
{\overline p}_1^{\ast}{\mathcal E} \longrightarrow {\mathcal I}_{\mathcal Z} 
\otimes {\mathcal L}^{-1} \longrightarrow 0
\end{equation*} 
where $\mathcal L$ is the line bundle ${\overline p}_2^{\ast}
{\overline p}_{2\ast}{\overline p}_1^{\ast}{\mathcal E}$ and $\mathcal Z$ is a 
closed subscheme of $\overline X$ of codimension $\geq 2$. 
In the case $c_1(\mathcal E)= -1$, ${\overline p}_1^{\ast}{\mathcal E}$ can 
be realized as an extension: 
\begin{equation*}
0 \longrightarrow 
({\pi}^{\ast}{\mathcal O}_{{\mathbb P}^3}(-1)\vert_{{\overline X}})
\otimes {\mathcal L} \longrightarrow 
{\overline p}_1^{\ast}{\mathcal E} \longrightarrow 
{\mathcal I}_{\mathcal Z}\otimes {\overline p}_1^{\ast}{\mathcal O}_Q(-1)\otimes 
({\pi}^{\ast}{\mathcal O}_{{\mathbb P}^3}(1)\vert_{{\overline X}})
\otimes {\mathcal L}^{-1} \longrightarrow 0
\end{equation*}
where $\mathcal L$ is the line bundle 
${\overline p}_2^{\ast}{\overline p}_{2\ast}({\overline p}_1^{\ast}{\mathcal E} 
\otimes ({\pi}^{\ast}{\mathcal O}_{{\mathbb P}^3}(1)\vert_{{\overline X}}))$ 
(see Cor.~\ref{C:nonstblonthecone}) and $\mathcal Z$ is a 
closed subscheme of $\overline X$ of codimension $\geq 2$. 

Since there exists no saturated subsheaf ${\overline{\mathcal E}}^{\, \prime}$ 
of ${\mathcal E}\vert_{Q\setminus \text{Sing}\, {\mathcal E}}$ such that, 
for every $y \in {\mathcal V}$, 
${\sigma}_y^{\ast}({\overline{\mathcal E}}^{\, \prime}\vert_Y) \simeq 
{\mathcal O}_{\widetilde Y}$ in the case $c_1(\mathcal E) = 0$, and such that 
${\sigma}_y^{\ast}({\overline{\mathcal E}}^{\, \prime}\vert_Y) \simeq 
{\pi}_y^{\ast}{\mathcal O}_{L_y}(-1)$ in the case $c_1(\mathcal E) = -1$, one 
deduces, from Prop.~\ref{P:standconstr}, the existence of a non-zero morphism: 
\begin{gather*}
{\mathcal L} \longrightarrow ({\Omega}_{{\widetilde p}_1}\vert_{{\overline X}}) 
\otimes {\mathcal I}_{\mathcal Z} \otimes {\mathcal L}^{-1} \  
\  (c_1 = 0)\\
({\pi}^{\ast}{\mathcal O}_{{\mathbb P}^3}(-1)\vert_{\overline X})\otimes {\mathcal L} 
\longrightarrow ({\Omega}_{{\widetilde p}_1}\vert_{{\overline X}}) 
\otimes {\mathcal I}_{\mathcal Z} \otimes 
{\overline p}_1^{\ast}{\mathcal O}_Q(-1)\otimes 
({\pi}^{\ast}{\mathcal O}_{{\mathbb P}^3}(1)\vert_{{\overline X}})\otimes 
{\mathcal L}^{-1}\  \  (c_1 = -1)
\end{gather*} 
The restriction of such a morphism to a general fibre of ${\overline p}_2 : 
{\overline X} \rightarrow {\mathcal V}$ is generically non-zero. One deduces, 
for a general $y \in {\mathcal V}$, the existence of a non-zero morphism: 
\begin{gather*}
{\mathcal O}_{\widetilde Y} \longrightarrow 
({\Omega}_{{\widetilde p}_1}\vert_{{\widetilde p}_2^{\, -1}(y)}) \otimes 
{\mathcal I}_{{\mathcal Z}_y} \  \  (\text{for}\  c_1(\mathcal E) = 0)\\
{\pi}_y^{\ast}{\mathcal O}_{L_y}(-1) \longrightarrow 
({\Omega}_{{\widetilde p}_1}\vert_{{\widetilde p}_2^{\, -1}(y)}) \otimes 
{\mathcal I}_{{\mathcal Z}_y}\otimes {\sigma}_y^{\ast}{\mathcal O}_Y(-1)\otimes 
{\pi}_y^{\ast}{\mathcal O}_{L_y}(1) \  \  (\text{for}\  c_1(\mathcal E) = -1)
\end{gather*} 
This means that, in the case $c_1(\mathcal E) = 0$, the rank 2 vector bundle 
${\Omega}_{{\widetilde p}_1}\vert_{{\widetilde p}_2^{\, -1}(y)}$ on 
${\widetilde p}_2^{\, -1}(y) \simeq {\widetilde Y}$ has a non-zero global 
section vanishing on the scheme ${\mathcal Z}_y$, and that, in the case 
$c_1(\mathcal E) = -1$, the vector bundle 
$({\Omega}_{{\widetilde p}_1}\vert_{{\widetilde p}_2^{\, -1}(y)}) \otimes 
{\sigma}_y^{\ast}{\mathcal O}_Y(-1)\otimes {\pi}_y^{\ast}{\mathcal O}_{L_y}(2) 
\simeq ({\Omega}_{{\widetilde p}_1}\vert_{{\widetilde p}_2^{\, -1}(y)}) 
\otimes {\mathcal O}_{\widetilde Y}(-C_y)$ 
has a global section vanishing on the scheme ${\mathcal Z}_y$. 

One uses, now, Cor.~\ref{C:omegap1}. In the case $c_1(\mathcal E) = 0$, it 
follows that ${\mathcal Z}_y$ is a subscheme of a fibre ${\widetilde L} =  
{\pi}_y^{-1}(\ell)$ of ${\pi}_y : {\widetilde Y} \rightarrow L_y$ (recall that 
${\mathcal Z}_y \cap C_y = \emptyset$). ${\sigma}_y : {\widetilde Y} 
\rightarrow Y$ maps $\widetilde L$ isomorphically onto a line $L \subset Y$. 
Since ${\mathcal Z}_y$ has length $c_2 := c_2(\mathcal E)$, one deduces that 
${\mathcal E}\vert_L \simeq {\mathcal O}_L(c_2) \oplus 
{\mathcal O}_L(-c_2)$. Consequently, through any general point $y \in Q$ 
passes a line $L\subset Q$ such that 
${\mathcal E}\vert_L \simeq {\mathcal O}_L(c_2) \oplus 
{\mathcal O}_L(-c_2)$. But this contradicts Lemma~\ref{L:dimjump}. 

Finally, in the case $c_1(\mathcal E) = -1$, it follows that ${\mathcal Z}_y 
= \emptyset$, i.e., that ${\sigma}_y^{\ast}{\mathcal E}$ can be realized as 
an extension: 
\begin{equation*}
0 \longrightarrow {\pi}_y^{\ast}{\mathcal O}_{L_y}(-1) \longrightarrow 
{\sigma}_y^{\ast}{\mathcal E} \longrightarrow  
{\sigma}_y^{\ast}{\mathcal O}_Y(-1) \otimes {\pi}_y^{\ast}{\mathcal O}_{L_y}(1) 
\longrightarrow 0 \, .
\end{equation*}
On the other hand, if $L$ is a line contained in $Y$ then, using the fact 
that there is an epimorphism ${\mathcal S} \rightarrow {\mathcal I}_L$, one 
gets an exact sequence: 
\begin{equation*}
0 \longrightarrow {\mathcal I}_{L,Y} \longrightarrow {\mathcal S}\vert_Y  
\longrightarrow {\mathcal I}_{L,Y} \longrightarrow 0\, .
\end{equation*}
One derives, as in the proof of Prop.~\ref{P:nonstblonthecone}(ii), the 
existence of an exact sequence: 
\begin{equation*}
0 \longrightarrow {\pi}_y^{\ast}{\mathcal O}_{L_y}(-1) \longrightarrow 
{\sigma}_y^{\ast}{\mathcal S} \longrightarrow  
{\sigma}_y^{\ast}{\mathcal O}_Y(-1) \otimes {\pi}_y^{\ast}{\mathcal O}_{L_y}(1) 
\longrightarrow 0 \, .
\end{equation*}  
Using the first part of Remark~\ref{R:h1el} one gets that:  
\begin{equation*}
\text{H}^1(\sigma_y^\ast{\mathcal O}_Y(1)\otimes \pi_y^\ast{\mathcal O}_{L_y}(-2)) 
\simeq \text{H}^1({\mathcal O}_{L_y}(-2) \oplus {\mathcal O}_{L_y}) 
\simeq {\mathbb C}\, .
\end{equation*} 
It follows that $\sigma_y^\ast{\mathcal E} \simeq \sigma_y^\ast{\mathcal S}$ 
hence ${\mathcal E}\vert_Y \simeq {\mathcal S}\vert_Y$.  
This implies, as in the last part of the proof of 
Remark~\ref{R:correctioneinsols}, that $\mathcal E$ is isomorphic to the 
spinor bundle $\mathcal S$. 
\end{proof}

\section*{Acknowledgements} 

Research for this work was initiated during the first named author's visit 
at the ``Universit\'{e} de Pau et des Pays de l'Adour'', in June 2010. 
The first named author is indebted to the University of Pau for hospitality 
and to CNRS, France, for financial support. Research of the second named 
author was partially supported by the ANR projects 
INTERLOW (ANR-09-JCJC-0097-01) and GEOLMI.


\begin{thebibliography}{99999}
\bibitem[Ba77]{Ba77} W. Barth: Some properties of stable vector bundles on 
               ${\mathbb P}^n$. Math. Ann. {\bf 226}, 125--150 (1977).
\bibitem[Be11]{Be11} O. Benoist: Le th\'{e}or\`{e}me de Bertini en 
               famille. Bull. Soc. Math. France {\bf 139}, 555--569 (2011),   
               also \texttt{arXiv:0911.1118}, (2009).
\bibitem[Br91]{Br91} V. Br\^{\i}nz\u{a}nescu: Algebraic 2-vector bundles on 
               ruled surfaces. Ann. Univ. Ferrara, Sez. VII, Sc. Mat.  
               {\bf 37}, 55--64 (1991). 
\bibitem[Bro83]{Bro83} J.E. Brossius: Rank-2 vector bundles on a ruled 
               surface I. Math. Ann. {\bf 265}, 155--168 (1983).  
\bibitem[Co92]{Co92} I. Coand\u{a}: On Barth's restriction theorem. 
               J. reine angew. Math. {\bf 428}, 97--110 (1992). 
\bibitem[ES84]{ES84} L. Ein and I. Sols: Stable vector bundles on quadric 
               hypersurfaces. Nagoya Math. J. {\bf 96}, 11--22 (1984).
\bibitem[Fa11]{Fa11} D. Faenzi: Even and odd instanton bundles on Fano 
               threefolds of Picard number 1. eprint   
               \texttt{arXiv:1109.3858}, (2011).  
\bibitem[Fl84]{Fl84} H. Flenner: Restrictions of semistable bundles on 
               projective varieties. Comment. Math. Helv.  
               {\bf 59}, 635--650 (1984). 
\bibitem[FHS80]{FHS80} O. Forster, A. Hirschowitz, and M. Schneider: Type 
               de scindage g\'{e}n\'{e}ralis\'{e} pour les fibr\'{e}s 
               stables. Proceedings, Nice 1979, Progress in Math. {\bf 7}, 
               65--81, Birkh\"{a}user, (1980).  
\bibitem[GM75]{GM75} H. Grauert and G. M\"{u}lich, Vektorb\"{u}ndel vom Rang 
               2 \"{u}ber dem $n$-dimensionalen komplex-projektiven Raum. 
               Manuscripta math. {\bf 16}, 75--100 (1975).  
\bibitem[Ha77]{Ha77} R. Hartshorne: Algebraic Geometry. Graduate Texts 
               in Math. {\bf 52}, Springer Verlag, (1977). 
\bibitem[Ha80]{Ha80} R. Hartshorne: Stable reflexive sheaves. 
               Math. Ann. {\bf 254}, 121--176 (1980).  
\bibitem[HL97]{HL97} D. Huybrechts and M. Lehn: The geometry of moduli 
               spaces of sheaves. Aspects of Mathematics {\bf E31}, Friedrich 
               Viehweg $\&$ Sohn, Braunschweig, (1997).
\bibitem[KO03]{KO03} P.I. Katsylo and G. Ottaviani: Regularity of the moduli 
               space of instanton bundles $\text{MI}_{{\mathbb P}^3}(5)$. 
               Transform. Groups {\bf 8}, 147--158 (2003). 
\bibitem[Ma76]{Ma76} M. Maruyama: Openess of a family of torsion free 
               sheaves. J. Math. Kyoto Univ. {\bf 16(3)}, 627--637 (1976). 
\bibitem[Mu99]{Mu99} D. Mumford: The Red Book of varieties and schemes. 
               Lecture Notes Math. {\bf 1358}, Second edition, Springer Verlag, 
               (1999). 
\bibitem[OSS80]{OSS80} C. Okonek, M. Schneider and H. Spindler:   
               Vector bundles on complex projective spaces.  
               Progress in Math. {\bf 3}, Birkh\"{a}user, (1980). 
\bibitem[Ott88]{Ott88} G. Ottaviani: Spinor bundles on quadrics.  
               Trans. Amer. Math. Soc. {\bf 307}, 301--316 (1988). 
\bibitem[OS94]{OS94} G. Ottaviani and M. Szurek: On moduli of stable 
               2-bundles with small Chern classes on $Q_3$. Annali Mat. 
               Pura Appl. (4) {\bf 167}, 191--241, With an appendix by 
               N. Manolache, (1994). 
\bibitem[SSW91]{SSW91} I. Sols, M. Szurek, and J.A. Wi\'{s}niewski: On Fano 
               bundles over a smooth 3-dimensional quadric. 
               Pacific J. Math. {\bf 148}, 153--160 (1991).
\bibitem[Va95]{Va95} J. Vall\`{e}s: Complexes inattendus de droites de sault. 
               C. R. Acad. Sci. Paris (I) {\bf 321}, 87--90 (1995).  
\bibitem[VdV72]{VdV72} A. Van de Ven: On uniform vector bundles. 
               Math. Ann. {\bf 195}, 245--248 (1972). 
\end{thebibliography}
\end{document}